\documentclass[12pt,letterpaper,american]{article}
\usepackage[T1]{fontenc}
\usepackage[latin9]{inputenc}
\usepackage{xcolor}
\usepackage{pdfcolmk}
\usepackage{babel}
\usepackage{mathrsfs}
\usepackage{amsmath}
\usepackage{amsthm}
\usepackage{amssymb}
\usepackage{graphicx}
\usepackage[all]{xy}
\PassOptionsToPackage{normalem}{ulem}
\usepackage{ulem}
\usepackage{nomencl}

\providecommand{\makenomenclature}{\makeglossary}
\makenomenclature
\usepackage[unicode=true,
 bookmarks=true,bookmarksnumbered=false,bookmarksopen=false,
 breaklinks=true,pdfborder={0 0 0},pdfborderstyle={},backref=false,colorlinks=false]
 {hyperref}
\hypersetup{pdftitle={Biorthogonal Approach to Infinite Dimensional Fractional Poisson Measure},
 pdfauthor={Jerome B. Bendong and Sheila M. Menchavez  and Jose Luis da Silva},
 pdfsubject={Non Gaussian Measures},
 pdfkeywords={Fractional Poisson measure,  Appell polynomials},
 pdfborderstyle=}

\makeatletter

\pdfpageheight\paperheight
\pdfpagewidth\paperwidth

\providecolor{lyxadded}{rgb}{0,0,1}
\providecolor{lyxdeleted}{rgb}{1,0,0}

\DeclareRobustCommand{\lyxsout}[1]{\ifx\\#1\else\sout{#1}\fi}

\theoremstyle{plain}
\newtheorem{thm}{\protect\theoremname}[section]
\theoremstyle{remark}
\newtheorem{rem}[thm]{\protect\remarkname}
\theoremstyle{plain}
\newtheorem{lem}[thm]{\protect\lemmaname}
\theoremstyle{plain}
\newtheorem{cor}[thm]{\protect\corollaryname}
\theoremstyle{definition}
\newtheorem{defn}[thm]{\protect\definitionname}
\theoremstyle{plain}
\newtheorem{prop}[thm]{\protect\propositionname}
\theoremstyle{definition}
\newtheorem{example}[thm]{\protect\examplename}

\usepackage{babel}
\usepackage{xcolor}
\usepackage{pdfcolmk}
\usepackage{bm}

\usepackage{datetime}
\usepackage{currfile}
\usepackage{pdfsync}
\usepackage{bbm,geometry}

\numberwithin{equation}{section}

\providecommand{\corollaryname}{Corollary}
\providecommand{\definitionname}{Definition}
\providecommand{\examplename}{Example}
\providecommand{\remarkname}{Remark}
\providecommand{\theoremname}{Theorem}

\newcommand{\vertiii}[1]{{\left\vert\kern-0.4ex\left\vert\kern-0.4ex\left\vert #1 
    \right\vert\kern-.4ex\right\vert\kern-0.4ex\right\vert}}

\providecommand{\propositionname}{Proposition}

\newcommand{\xyR}[1]{
  \xydef@\xymatrixrowsep@{#1}}
\newcommand{\xyC}[1]{
  \xydef@\xymatrixcolsep@{#1}}

\makeatother

\providecommand{\corollaryname}{Corollary}
\providecommand{\definitionname}{Definition}
\providecommand{\examplename}{Example}
\providecommand{\lemmaname}{Lemma}
\providecommand{\propositionname}{Proposition}
\providecommand{\remarkname}{Remark}
\providecommand{\theoremname}{Theorem}

\begin{document}
\title{A Biorthogonal Approach to the Infinite Dimensional Fractional Poisson
Measure}
\author{\textbf{Jerome B. Bendong}\\
 DMS, MSU-Iligan Institute of Technology,\\
 Tibanga, 9200 Iligan City, Philippines\\
 Email: jerome.bendong@g.msuiit.edu.ph\and \textbf{Sheila M. Menchavez}\\
 DMS, MSU-Iligan Institute of Technology,\\
 Tibanga, 9200 Iligan City, Philippines\\
 Email: sheila.menchavez@g.msuiit.edu.ph\and \textbf{Jos{\'e} Lu{\'\i}s
da Silva}\\
Faculdade de Ci{\^e}ncias Exatas e da Engenharia,\\
CIMA, Universidade da Madeira,\\
Campus Universit{\'a}rio da Penteada,\\
9020-105 Funchal, Portugal\\
 Email: joses@staff.uma.pt}
\date{\today}
\maketitle
\begin{abstract}
In this paper we use a biorthogonal approach to the analysis of the
infinite dimensional fractional Poisson measure $\pi_{\sigma}^{\beta}$,
$0<\beta\leq1$, on the dual of Schwartz test function space $\mathcal{D}'$.
The Hilbert space $L^{2}(\pi_{\sigma}^{\beta})$ of complex-valued
functions is described in terms of a system of generalized Appell
polynomials $\mathbb{P}^{\sigma,\beta,\alpha}$ associated to the
measure $\pi_{\sigma}^{\beta}$. The kernels $C_{n}^{\sigma,\beta}(\cdot)$,
$n\in\mathbb{N}_{0}$, of the monomials may be expressed in terms
of the Stirling operators of the first and second kind as well as
the falling factorials in infinite dimensions. Associated to the system
$\mathbb{P}^{\sigma,\beta,\alpha}$, there is a generalized dual Appell
system $\mathbb{Q}^{\sigma,\beta,\alpha}$ that is biorthogonal to
$\mathbb{P}^{\sigma,\beta,\alpha}$. The test and generalized function
spaces associated to the measure $\pi_{\sigma}^{\beta}$ are completely
characterized using an integral transform as entire functions. \\
 \\
 \textbf{Keywords}: fractional Poisson measure, generalized Appell
system, Wick exponential, test functions, generalized functions, Stirling
operators, $S$-transform.
\end{abstract}
\tableofcontents{}

\section{Introduction}

\label{sec:Introduction}In this paper we develop a biorthogonal approach
to the analysis of the infinite dimensional fractional Poisson measure
(fPm) on the configuration space $\Gamma$ or over $\mathcal{D}'$
(the dual of the Schwartz test function space $\mathcal{D}$). As
a special case of a non-Gaussian measure (for which this biorthogonal
approach was developed in \cite{ADKS96,KSWY95,KdSS98}) the fPm revealed
an interesting connection with the Stirling operators and falling
factorials in the context of infinite dimensional analysis introduced
recently in \cite{Finkelshtein2022}.

To describe our results more precisely, let us recall that there are
different ways to introduce a total set of orthogonal polynomials
in the Hilbert space of square integrable functions with respect to
(wrt) a probability measure. For example, applying the Gram-Schmidt
method to an independent sequence of functions or using generating
functions. In the case at hand, that is, the fPm $\pi_{\sigma}^{\beta}$
($0<\beta\le1$, $\sigma$ a non-degenerate and non-atomic measure
in $\mathbb{R}^{d}$), we have chosen the generating function procedure
because the Gram-Schmidt method is not practical. In addition, the
generating function is picked in a way such that at $\beta=1$, we
recover the classical Charlier polynomials, that is, $\pi_{\sigma}^{1}$
coincides with the standard Poisson measure $\pi_{\sigma}$ on $\Gamma$,
see \cite{AKR97a} for more details. In explicit, given the map
\[
\alpha:\mathcal{D}_{\mathbb{C}}\longrightarrow\mathcal{D}_{\mathbb{C}},\;\varphi\mapsto\alpha(\varphi)(x):=\log(1+\varphi(x)),\quad x\in\mathbb{R}^{d},
\]
we define the modified Wick exponential 
\[
\mathrm{e}_{\pi_{\sigma}^{\beta}}(\alpha(\varphi);w):=\frac{\exp(\langle w,\alpha(\varphi)\rangle)}{l_{\pi_{\sigma}^{\beta}}(\alpha(\varphi))}=\sum_{n=0}^{\infty}\frac{1}{n!}\langle C_{n}^{\sigma,\beta}(w),\varphi^{\otimes n}\rangle,\quad w\in\mathcal{D}'_{\mathbb{C}},
\]
where $\varphi$ is properly chosen from a neighborhood of zero in
$\mathcal{D}_{\mathbb{C}}$. The monomials $\langle C_{n}^{\sigma,\beta}(w),\varphi^{\otimes n}\rangle$,
$n\in\mathbb{N}_{0}$ generates a system of polynomials $\mathbb{P}^{\sigma,\beta,\alpha}$
which forms a total set in the space $L^{2}(\pi_{\sigma}^{\beta})$
of square $\pi_{\sigma}^{\beta}$-integrable complex functions. The
kernels $C_{n}^{\sigma,\beta}(\cdot)$, $n\in\mathbb{N}_{0}$, possess
certain remarkable properties involving the Stirling operators of
the first and second kind as well as the falling factorials $(w)_{n}$,
$w\in\mathcal{D}'_{\mathbb{C}}$ introduced in \cite{Finkelshtein2022}.
We refer to Proposition~\ref{prop:generalized-appell-polynomials-inf-dim}
and Appendix \ref{sec:Stirling-Operators} for more details and results.
Other choice of generating functions like $\mathrm{e}_{\pi_{\sigma}^{\beta}}(\varphi;\cdot)$
is also possible (see the beginning of Section \ref{sec:Appell-System}),
but at $\beta=1$, the corresponding system of polynomials do not
coincide with the classical Charlier polynomials. Thus, our natural
choice goes to the modified Wick exponential generating function $\mathrm{e}_{\pi_{\sigma}^{\beta}}(\alpha(\varphi);\cdot)$.

On the other hand, the construction of the generalized dual Appell
system $\mathbb{Q}^{\sigma,\beta,\alpha}$ turns out to be very appealing
since it involves a differential operator of infinite order on the
space of polynomials $\mathcal{P}(\mathcal{D}')$ over $\mathcal{D}'$
and the adjoints of the Stirling operators. This careful choice of
the system $\mathbb{Q}^{\sigma,\beta,\alpha}$ leads us to the so-called
biorthogonal property between the two systems $\mathbb{P}^{\sigma,\beta,\alpha}$
and $\mathbb{Q}^{\sigma,\beta,\alpha}$, see Theorem \ref{thm:biorthogonal-property-inf-dim}.

The generalized Appell system $\mathbb{A}^{\sigma,\beta,\alpha}:=(\mathbb{P}^{\sigma,\beta,\alpha},\mathbb{Q}^{\sigma,\beta,\alpha})$
is used to introduce a family of test function spaces $(\mathcal{N})_{\pi_{\sigma}^{\beta}}^{\kappa}$,
$0\le\kappa\le1$, which are nuclear spaces and continuously embedded
in $L^{2}(\pi_{\sigma}^{\beta})$. The dual space of $(\mathcal{N})_{\pi_{\sigma}^{\beta}}^{\kappa}$
is given by the general duality theory as $(\mathcal{N})_{\pi_{\sigma}^{\beta}}^{-\kappa}$.
In this way, we obtain the chain of continuous embeddings
\[
(\mathcal{N})_{\pi_{\sigma}^{\beta}}^{\kappa}\subset L^{2}(\pi_{\sigma}^{\beta})\subset(\mathcal{N})_{\pi_{\sigma}^{\beta}}^{-\kappa}.
\]
A typical example of a test function is the modified Wick exponential
$\mathrm{e}_{\pi_{\sigma}^{\beta}}(\alpha(\varphi);\cdot)\in(\mathcal{N})_{\pi_{\sigma}^{\beta}}^{\kappa}$
(see Example \ref{exa:normalized-exp-as-test-function}) given as
a convergent series in terms of the system $\mathbb{P}^{\sigma,\beta,\alpha}$
while a particular element in $(\mathcal{N})_{\pi_{\sigma}^{\beta}}^{-1}$
is given by the generalized Radon-Nikodym derivative $\rho_{\pi_{\sigma}^{\beta}}^{\alpha}(w,\cdot)$,
$w\in\mathcal{N}'_{\mathbb{C}}$. Moreover, the generalized function
$\rho_{\pi_{\sigma}^{\beta}}^{\alpha}(w,\cdot)$ plays the role of
the generating function of the system $\mathbb{Q}^{\sigma,\beta,\alpha}$,
that is,
\[
\rho_{\pi_{\sigma}^{\beta}}^{\alpha}(w,\cdot)=\sum_{k=0}^{\infty}\frac{1}{k!}Q_{k}^{\text{\ensuremath{\pi_{\sigma}^{\beta},\alpha}}}((-w)_{k}).
\]
The spaces $(\mathcal{N})_{\pi_{\sigma}^{\beta}}^{\pm\kappa}$ may
be characterized in terms of an integral transform, called the $S_{\pi_{\sigma}^{\beta}}$-transform.
It turns out that all these spaces $(\mathcal{N})_{\pi_{\sigma}^{\beta}}^{\pm\kappa}$,
$0\le\kappa\le1$, are universal in the sense that the $S_{\pi_{\sigma}^{\beta}}$-transform
of their elements are entire functions (for $0\le\kappa<1)$ or holomorphic
functions ($\kappa=1$) and independent of the measure $\pi_{\sigma}^{\beta}$,
see Theorem~\ref{thm:characterizations}. This feature is well known
in non-Gaussian analysis.

The paper is organized as follows. In Section \ref{sec:Nuclear-spaces},
we recall some known concepts of nuclear spaces and their tensor products.
As a motivation to the generalization of fPm to infinite dimensions,
we discuss its finite dimensional version in Section \ref{sec:FPm-finite-dim}.
We show that the monic polynomials $C_{n}^{\beta}(x)$, $n\in\mathbb{N}_{0}$,
obtained using Gram-Schmidt orthogonalization process to the monomials
$x^{n}$, $n\in\mathbb{\mathbb{N}}_{0}$ are orthogonal in $L^{2}(\pi_{\vec{\lambda},\beta}^{2})$
($\pi_{\vec{\lambda},\beta}^{2}$ is the Poisson measure in 2-dimensions)
if, and only if, $\beta=1$. In Section \ref{sec:Inf-Dim-fPm}, we
define the fPm $\pi_{\sigma}^{\beta}$ in infinite dimensions as a
probability measure on $(\mathcal{D}',\mathcal{C}_{\sigma}(\mathcal{D}'))$,
where $\mathcal{C}_{\sigma}(\mathcal{D}')$ is the $\sigma$-algebra
generated by the cylinder sets. We also discuss the concept of configuration
space $\Gamma$ and then using the Kolmogorov extension theorem we
define a unique measure $\pi_{\sigma}^{\beta}$ on the configuration
space $(\Gamma,\mathcal{B}(\Gamma))$ whose characteristic function
coincides with that of $\pi_{\sigma}^{\beta}$ on the distribution
space $\mathcal{D}'$. In Section \ref{sec:Appell-System}, we introduce
the generalized Appell system associated with the fPm $\pi_{\sigma}^{\beta}$.
This includes the system of generalized Appell polynomials and the
dual Appell system which are biorthogonal with respect to the fPm
$\pi_{\sigma}^{\beta}$. Finally in Section \ref{sec:test-and-generalized-function-spaces},
we construct the test and generalized function spaces associated to
the fPm $\pi_{\sigma}^{\beta}$ and provide some properties as well
as its characterization theorems.

For completeness, in the Appendices \ref{sec:Kolmogorov-extension-theorem-on-config-space}--\ref{sec:alternative-proof-biorthogonal-property}
we provide certain concepts and results already known in the literature,
particularly, the Kolmogorov extension theorem on the configuration
space and Stirling operators in infinite dimensions.

\section{Tensor Powers of Nuclear Spaces}

\label{sec:Nuclear-spaces}We first consider nuclear Fr\'{e}chet
spaces (i.e., a complete metrizable locally convex space) that may
be characterized in terms of projective limits of a countable number
of Hilbert spaces, see e.g., \cite{BK88}, \cite{BSU96}, \cite{GV68},
\cite{HKPS93} and \cite{O94} for more details and proofs.

Let $\mathcal{H}$ be a real separable Hilbert space with inner product
$(\cdot,\cdot)$ and corresponding norm $|\cdot|$. Consider a family
of real separable Hilbert space $\mathcal{H}_{p}$, $p\in\mathbb{N}$
with Hilbert norm $|\cdot|_{p}$ such that the space $\bigcap_{p\in\mathbb{N}}\mathcal{H}_{p}$
is dense in each $\mathcal{H}_{p}$, and 
\[
\dots\subset\mathcal{H}_{p}\subset\dots\subset\mathcal{H}_{1}\subset\mathcal{H}
\]
with the corresponding system of norms being ordered, i.e.,

\[
|\cdot|\leq|\cdot|_{1}\leq\dots\leq|\cdot|_{p}\leq\dots\quad p\in\mathbb{N}.
\]

Now we assume that the space $\mathcal{N}=\bigcap_{p\in\mathbb{N}}\mathcal{H}_{p}$
is nuclear (i.e., for each $p\in\mathbb{N}$ there is a $q>p$ such
that the canonical embedding $\mathcal{H}_{q}\hookrightarrow\mathcal{H}_{p}$
is of Hilbert-Schmidt class) and on $\mathcal{N}$ we fix the \emph{projective
limit topology}, i.e., the coarsest topology on $\mathcal{N}$ with
respect to which each canonical embedding $\mathcal{N}\hookrightarrow\mathcal{H}_{p}$,
$p\in\mathbb{N}$, is continuous.  With respect to this topology,
$\mathcal{N}$ is a Fr\'{e}chet space and we use the notation
\[
\mathcal{N}=\underset{p\in\mathbb{N}}{\mathrm{pr\,lim}}\mathcal{H}_{p}
\]
to denote the space $\mathcal{N}$ endowed with the corresponding
projective limit topology. Such a topological space is called a \emph{projective
limit} or a \emph{countable limit of the family} $(\mathcal{H}_{p})_{p\in\mathbb{N}}$.

Let us denote by $\mathcal{H}_{-p}$, $p\in\mathbb{N}$, the dual
of $\mathcal{H}_{p}$ with respect to the space $\mathcal{H}$, with
the corresponding Hilbert norm $|\cdot|_{-p}$. By the general duality
theory, the dual space $\mathcal{N}'$ of $\mathcal{N}$ with respect
to $\mathcal{H}$ can then be written as
\[
\mathcal{N}':=\bigcup_{p\in\mathbb{N}}\mathcal{H}_{-p}
\]
with the \emph{inductive limit topology}, i.e., the finest topology
on $\mathcal{N}'$ with respect to which all the embeddings $\mathcal{H}_{-p}\hookrightarrow\mathcal{N}'$
are continuous. This topological space is denoted by
\[
\mathcal{N}'=\underset{p\in\mathbb{N}}{\mathrm{ind\,lim}}\mathcal{H}_{-p}
\]
and is called an \emph{inductive limit of the family} $(\mathcal{H}_{-p})_{p\in\mathbb{N}}$.
In this way we have obtained the chain of spaces 
\[
\mathcal{N}\mathcal{\subset H}\subset\mathcal{N}'
\]
called a \emph{nuclear triple} or \emph{Gelfand triple}. The dual
pairing $\langle\cdot,\cdot\rangle$ between $\mathcal{N}$ and $\mathcal{N}'$
is then realized as an extension of the inner product $(\cdot,\cdot)$
on $\mathcal{H}$, i.e., 
\[
\langle g,\xi\rangle=(g,\xi),\quad g\in\mathcal{H},\xi\in\mathcal{N}.
\]
The tensor product of the Hilbert spaces $\mathcal{H}_{p}$, $p\in\mathbb{N}$,
is denoted by $\mathcal{H}_{p}^{\otimes n}$. We keep the notation
$|\cdot|_{p}$ for the Hilbert norm on this space. The subspace of
$\mathcal{H}_{p}^{\otimes n}$ of symmetric elements is denoted by
$\mathcal{H}_{p}^{\hat{\otimes}n}$. The \emph{$n$}-th tensor power
$\mathcal{N}^{\otimes n}$ of $\mathcal{N}$ and the \emph{$n$}-th
symmetric tensor power $\mathcal{N}^{\hat{\otimes}n}$ of $\mathcal{N}$
are the nuclear Fr\'{e}chet spaces given by 
\[
\mathcal{N}^{\otimes n}:=\underset{p\in\mathbb{N}}{\mathrm{pr\,lim}}\mathcal{H}_{p}^{\otimes n}\mathrm{\quad and}\quad\mathcal{N}^{\hat{\otimes}n}:=\underset{p\in\mathbb{N}}{\mathrm{pr\,lim}}\mathcal{H}_{p}^{\hat{\otimes}n}.
\]
Furthermore, if $\mathcal{H}_{-p}^{\otimes n}$ (resp., $\mathcal{H}_{-p}^{\hat{\otimes}n}$
) denotes the dual space of $\mathcal{H}_{p}^{\otimes n}$ (resp.,
$\mathcal{H}_{p}^{\hat{\otimes}n}$) with respect to $\mathcal{H}^{\otimes n}$,
then the dual space $(\mathcal{N}{}^{\otimes n})'$ of $\mathcal{N}^{\otimes n}$
with respect to $\mathcal{H}^{\otimes n}$ and the dual space $(\mathcal{N}{}^{\hat{\otimes}n})'$
of $\mathcal{N}^{\hat{\otimes}n}$ with respect to $\mathcal{H}^{\otimes n}$
can be written as 
\[
(\mathcal{N}{}^{\otimes n})'=\underset{p\in\mathbb{N}}{\mathrm{ind\,lim}}\mathcal{H}_{-p}^{\otimes n}\mathrm{\quad and}\quad(\mathcal{N}{}^{\hat{\otimes}n})'=\underset{p\in\mathbb{N}}{\mathrm{ind\,lim}}\mathcal{H}_{-p}^{\hat{\otimes}n},
\]
respectively. As before we use the notation $|\cdot|_{-p}$ for the
norm on $\mathcal{H}_{-p}^{\otimes n}$, $p\in\mathbb{N}$, and $\langle\cdot,\cdot\rangle$
for the dual pairing between $(\mathcal{N}{}^{\otimes n})'$ and $\mathcal{N}^{\otimes n}$.
Thus we have defined the nuclear triples 
\[
\mathcal{N}^{\otimes n}\subset\mathcal{H}^{\otimes n}\subset(\mathcal{N}{}^{\otimes n})'\quad\mathrm{and}\quad\mathcal{N}^{\hat{\otimes}n}\subset\mathcal{H}^{\hat{\otimes}n}\subset(\mathcal{N}{}^{\hat{\otimes}n})'.
\]

To all the real spaces in this section, we may also consider their
complexifications which will be distinguished by a subscript $\mathbb{C}$,
i.e., the complexification of $\text{\ensuremath{\mathcal{H}}}$ is
$\text{\ensuremath{\mathcal{H}}}_{\mathbb{C}}$ and so on. This means
that for $h\in\text{\ensuremath{\mathcal{H}}}_{\mathbb{C}}$, we have
$h=h_{1}+ih_{2}$ where $h_{1},h_{2}\in\text{\ensuremath{\mathcal{H}}}$.

Let us now introduce spaces of entire functions which will be used
later in the characterization theorems in Section \ref{sec:test-and-generalized-function-spaces}.
Let $\mathcal{E}_{2^{-l}}^{k}(\mathcal{H}_{-p,\mathbb{C}})$ denote
the set of all entire functions on $\mathcal{H}_{-p,\mathbb{C}}$
of growth $k\in[1,2]$ and type $2^{-l}$, $p,l\in\mathbb{Z}$. This
is a linear space with norm 
\[
n_{p,l,k}(\varphi)=\sup_{w\in\mathcal{H}_{-p,\mathbb{C}}}|\varphi(w)|\exp(-2^{-l}|z|_{-p}^{k}),\;\varphi\in\mathcal{E}_{2^{-l}}^{k}(\mathcal{H}_{-p,\mathbb{C}}).
\]
The space of entire functions on $\mathcal{N}'_{\mathbb{C}}$ of growth
$k$ and minimal type is naturally introduced by
\[
\mathcal{E}_{\min}^{k}(\mathcal{N}'_{\mathbb{C}}):=\underset{p,l\in\mathbb{N}}{\mathrm{pr\,lim}}\mathcal{E}_{2^{-l}}^{k}(\mathcal{H}_{-p,\mathbb{C}}),
\]
see e.g., \cite{Ko80a,BK88}. We will also need the space of entire
functions on $\mathcal{N}_{\mathbb{C}}$ of growth $k$ and finite
type given by 
\[
\mathcal{E}_{\max}^{k}(\mathcal{N}{}_{\mathbb{C}}):=\underset{p,l\in\mathbb{N}}{\mathrm{ind\,lim}}\mathcal{E}_{2^{l}}^{k}(\mathcal{H}_{p,\mathbb{C}}).
\]

\section{Finite Dimensional Fractional Poisson Measure}

\label{sec:FPm-finite-dim}In this section we discuss the finite dimensional
version of the fractional Poisson measure as a motivation to its generalization
to infinite dimensions. The one dimensional version of the fractional
Poisson analysis was studied in \cite{Bendong2022}.

At first we introduce the Mittag-Leffler function $E_{\beta}$ with
parameter $\beta\in(0,1]$. The Mittag-Leffler function is an entire
function defined on the complex plane by the power series 
\begin{equation}
E_{\beta}(z):=\sum_{n=0}^{\infty}\frac{z^{n}}{\Gamma(\beta n+1)},\quad z\in\mathbb{C}.\label{eq:ML-function}
\end{equation}
The Mittag-Leffler function plays the same role for the fPm as the
exponential function plays for Poisson measure. Note that for $\beta=1$
we have $E_{1}(z)=\mathrm{e}^{z}$.

For any $0<\beta\le1$, the fPm $\pi_{\lambda,\beta}$ on $\mathbb{N}_{0}$
(or $\mathbb{R})$ with rate $\lambda>0$ is defined for any $B\in\mathscr{P}(\mathbb{N}_{0})$
by
\[
\pi_{\lambda,\beta}(B):=\sum_{k\in B}\frac{\lambda^{k}}{k!}E_{\beta}^{(k)}(-\lambda),
\]
where $E_{\beta}^{(k)}(z):=\frac{d^{k}}{dz^{k}}E_{\beta}(z)$ is the
$k$-th derivative of the Mittag-Leffler $E_{\beta}$ function. In
particular, if $B=\{k\}\in\mathscr{P}(\mathbb{N}_{0})$, $k\in\mathbb{N}_{0}$,
we obtain
\[
\pi_{\lambda,\beta}(\{k\}):=\frac{\lambda^{k}}{k!}E_{\beta}^{(k)}(-\lambda).
\]
The Laplace transform of the measure $\pi_{\lambda,\beta}$ is given
for any $z\in\mathbb{C}$ by 
\begin{equation}
l_{\pi_{\lambda,\beta}}(z)=\int_{\mathbb{R}}\mathrm{e}^{zx}\,\mathrm{d}\pi_{\lambda,\beta}(x)=\sum_{k=0}^{\infty}\frac{\big(\mathrm{e}^{z}\lambda\big)^{k}}{k!}E_{\beta}^{(k)}\big(-\lambda\big)=E_{\beta}\big(\lambda(\mathrm{e}^{z}-1)\big).\label{eq:LT-fPm}
\end{equation}

\begin{rem}
The measure $\pi_{\lambda t^{\beta},\beta}$ corresponds to the marginal
distribution of the fractional Poisson process $N_{\lambda,\beta}=(N_{\lambda,\beta}(t))_{t\ge0}$
with parameter $\lambda t^{\beta}>0$ defined on a probability space
$(\Omega,\mathcal{F},P)$. Thus, we obtain
\[
\pi_{\lambda t^{\beta},\beta}(\{k\})=P(N_{\lambda,\beta}(t)=k)=\frac{(\lambda t^{\beta})^{k}}{k!}E_{\beta}^{(k)}(-\lambda t^{\beta}),\quad k\in\mathbb{N}_{0}.
\]
\end{rem}

\begin{rem}
The fractional Poisson process $N_{\lambda,\beta}$ was proposed by
O.\ N.\ Repin and A.\ I.\ Saichev \cite{Repin-Saichev00}. Since
then, it was studied by many authors see for example \cite{L03,MGS04,Mainardi-Gorenflo-Vivoli-05,Gorenflo2015,Uchaikin2008,Beghin:2009fi,Politi-Kaizoji-2011,Meerschaert2011,Biard2014}
and references therein.
\end{rem}

A remarkable property of the fPm is that $\pi_{\lambda,\beta}$ is
given as a mixture of Poisson measures with respect to a probability
measure $\nu_{\beta}$ on $\mathbb{R}_{+}:=[0,\infty)$. That probability
measure $\nu_{\beta}$ is absolutely continuous with respect to the
Lebesgue measure on $\mathbb{R}_{+}$ with a probability density $W_{-\beta,1-\beta}$,
that is, the Wright function. The Laplace transform of the measure
$\nu_{\beta}$ (or its density $W_{-\beta,1-\beta}$) is given by
\begin{equation}
\int_{0}^{\infty}\mathrm{e}^{-\tau z}\,\mathrm{d}\nu_{\beta}(\tau)=\int_{0}^{\infty}\mathrm{e}^{-z\tau}W_{-\beta,1-\beta}(\tau)\,\mathrm{d}\tau=E_{\beta}(-z),\label{eq:monotonicity-Mittag}
\end{equation}
for any $z\in\mathbb{C}$ such that $\mathrm{Re}(z)\geq0$, see \cite[Cor.~A.5]{GJRS14}.
Equation \eqref{eq:monotonicity-Mittag} is called the complete monotonicity
property of the Mittag-Leffler function, see \cite{Pollard48}. More
precisely, we have the following lemma.
\begin{lem}
\label{lem:mixture-1d}For $0<\beta\leq1$, the fPm $\pi_{\lambda,\beta}$
is an integral (or mixture) of Poisson measure $\pi_{\lambda}$ with
respect to the probability measure $\nu_{\beta}$, i.e.,
\begin{equation}
\pi_{\lambda,\beta}=\int_{0}^{\infty}\pi_{\lambda\tau}\,\mathrm{d\nu_{\beta}(\tau),}\quad\forall\lambda>0.\label{eq:fPm-mixture}
\end{equation}
\end{lem}

\begin{proof}
For $\beta=1$, we have $\nu_{1}=\delta_{1}$, the Dirac measure at
1, and the result is clear. For $0<\beta<1$, we denote the right
hand side of \eqref{eq:fPm-mixture} by $\mathrm{\mu}:=\int_{0}^{\infty}\mathrm{\pi}_{\lambda\tau}W_{-\beta,1-\beta}(\tau)\,\mathrm{d}\tau$.
We compute the Laplace transform of $\mu$ and use Fubini's theorem
to obtain 
\begin{align*}
\int_{0}^{\infty}\mathrm{e}^{zx}\,\mathrm{d}\mu(x) & =\int_{0}^{\infty}\mathrm{e}^{zx}\int_{0}^{\infty}\mathrm{d\pi}_{\lambda\tau}(x)W_{-\beta,1-\beta}(\tau)\,\mathrm{d}\tau\\
 & =\int_{0}^{\infty}\left(\int_{0}^{\infty}\mathrm{e}^{zx}\mathrm{d\pi}_{\lambda\tau}(x)\right)W_{-\beta,1-\beta}(\tau)\,\mathrm{d}\tau\\
 & =\int_{0}^{\infty}\mathrm{e}^{\tau\lambda(\mathrm{e}^{z}-1)}W_{-\beta,1-\beta}(\tau)\,\mathrm{d}\tau\\
 & =E_{\beta}(\lambda(e^{z}-1)).
\end{align*}
Thus, we conclude that the Laplace transforms of $\mu$ and $\pi_{\lambda,\beta}$
(cf.~\eqref{eq:LT-fPm}) coincide. The result follows by the uniqueness
of the Laplace transform.
\end{proof}
\begin{thm}[Moments of $\pi_{\lambda,\beta},$ cf.~\cite{L09}]
\label{thm:fPmm}The fPm $\pi_{\lambda,\beta}$ has moments of all
order. More precisely, the $n$-th moment of the measure $\pi_{\lambda,\beta}$
is given by 
\begin{equation}
m_{\lambda,\beta}(n):=\int_{\mathbb{R}}x^{n}\,\mathrm{d}\pi_{\lambda,\beta}(x)=\sum_{m=0}^{n}\frac{m!}{\Gamma(m\beta+1)}S(n,m)\lambda^{m},\label{cnn}
\end{equation}
where $S(n,m)$ is the Stirling number of the second kind.
\end{thm}

Here are the first few moments of the measure $\pi_{\lambda,\beta}$:
\begin{center}
$\begin{array}{rcl}
m_{\lambda,\beta}(0) & = & 1,\\
m_{\lambda,\beta}(1) & = & {\displaystyle \frac{\lambda}{\Gamma(\beta+1)},}\\
m_{\lambda,\beta}(2) & = & {\displaystyle \frac{\lambda}{\Gamma(\beta+1)}+\frac{2\lambda^{2}}{\Gamma(2\beta+1)},}\\
m_{\lambda,\beta}(3) & = & {\displaystyle \frac{\lambda}{\Gamma(\beta+1)}+\frac{6\lambda^{2}}{\Gamma(2\beta+1)}+\frac{6\lambda^{3}}{\Gamma(3\beta+1)}.}
\end{array}$
\par\end{center}

\begin{flushleft}
When $\beta=1$, these moments become the moments of the Poisson measure.
\par\end{flushleft}

In addition to the Poisson measure $\pi_{\lambda}$ and fPm $\pi_{\lambda,\beta}$
in $\mathbb{N}_{0}$ we also need the two dimensional version of both
of these measures in $\mathbb{N}_{0}^{2}$ or $\mathbb{R}^{2}$, the
reason for that is clear after Corollary \ref{cor:orthogonal-for-beta1}.
 The $d$-dimensional Poisson measure is given by 
\[
\pi_{\vec{\lambda}}^{d}(\{k_{1},\dots,k_{d}\})=\prod_{i=1}^{d}\frac{\lambda_{i}^{k_{i}}}{k_{i}!}\mathrm{e}^{-\lambda_{i}}.
\]
The Laplace transform of $\pi_{\vec{\lambda}}^{2}$ is given by
\begin{equation}
l_{\pi_{\vec{\lambda}}^{2}}(z)=\int_{\mathbb{R}^{2}}\mathrm{e}^{(x,s)}\,\mathrm{d}\pi_{\vec{\lambda}}^{2}(x)=\exp\big(\lambda_{1}(\mathrm{e}^{s_{1}}-1)+\lambda_{2}(\mathrm{e}^{s_{2}}-1)\big)\label{eq:LT-Pm-dim2}
\end{equation}
where $s=(s_{1},s_{2})\in\mathbb{R}^{2}$. For any $0<\beta\le1$,
$\vec{\lambda}\in(\mathbb{R}_{+}^{*})^{2},$ then a possible fractional
generalization of $\pi_{\vec{\lambda}}^{2}$, denoted by $\pi_{\vec{\lambda},\beta}^{2}$,
is given, via its Laplace transform, by replacing the first exponential
function on the right hand side of \eqref{eq:LT-Pm-dim2} by the Mittag-Leffler
function. More precisely, the Laplace transform of $\pi_{\vec{\lambda},\beta}^{2}$
is given by 
\begin{equation}
l_{\pi_{\vec{\lambda},\beta}^{2}}(s)=\int_{\mathbb{R}^{2}}\mathrm{e}^{(x,s)}\,\mathrm{d}\pi_{\vec{\lambda},\beta}^{2}(x)=E_{\beta}\big(\lambda_{1}(\mathrm{e}^{s_{1}}-1)+\lambda_{2}(\mathrm{e}^{s_{2}}-1)\big),\label{eq:LT-fPm-2d}
\end{equation}
where $s=(s_{1},s_{2})\in\mathbb{R}^{2}$.

The moments of the measure $\pi_{\vec{\lambda},\beta}^{2}$, denoted
by $m_{\vec{\lambda},\beta}^{2}(n_{1},n_{2})$, can be obtained by
applying $\frac{d^{n_{1}}}{ds_{1}^{n_{1}}}\frac{d^{n_{2}}}{ds_{2}^{n_{2}}}$,
$n_{1},n_{2}\in\mathbb{N}_{0}$, to Equation \eqref{eq:LT-fPm-2d}
and then evaluating at $s_{1}=s_{2}=0$. As an example, here we compute
the moments $m_{\vec{\lambda},\beta}^{2}(1,1)$ and $m_{\vec{\lambda},\beta}^{2}(1,2)$
of the measure $\pi_{\vec{\lambda},\beta}^{2}$ needed later on:

\begin{align*}
m_{\vec{\lambda},\beta}^{2}(1,1)=\int_{\mathbb{R}^{2}}x_{1}x_{2}\,\mathrm{d}\pi_{\vec{\lambda},\beta}^{2}(x_{1,}x_{2}) & =\frac{2\lambda_{1}\lambda_{2}}{\Gamma(2\beta+1)},\\
m_{\vec{\lambda},\beta}^{2}(1,2)=\int_{\mathbb{R}^{2}}x_{1}x_{2}^{2}\,\mathrm{d}\pi_{\vec{\lambda},\beta}^{2}(x_{1,}x_{2}) & =\frac{2\lambda_{1}\lambda_{2}}{\Gamma(2\beta+1)}+\frac{6\lambda_{1}\lambda_{2}^{2}}{\Gamma(3\beta+1)}.
\end{align*}

We apply the Gram-Schmidt orthogonalization process to the monomials
$x^{n}$, $n\in\mathbb{\mathbb{N}}_{0},$ to obtain monic polynomials
$C_{n}^{\beta}(x)$ with $\deg C_{n}^{\beta}(x)=n$ with respect to
the inner product
\[
(p,q)_{\pi_{\lambda,\beta}}:=\int_{\mathbb{R}}p(x)q(x)\,\mathrm{d}\pi_{\lambda,\beta}(x).
\]
These polynomials are determined by the moments of the measure $\pi_{\lambda,\beta}$.
The first few of these polynomials are given by
\begin{align*}
C_{0}^{\beta}(x) & =1,\\
C_{1}^{\beta}(x) & =x-(x,C_{0}^{\beta})_{\pi_{\lambda,\beta}}C_{0}^{\beta}(x)=x-m_{\lambda,\beta}(1),\\
C_{2}^{\beta}(x) & =x^{2}-(x^{2},C_{0}^{\beta})_{\pi_{\lambda,\beta}}C_{0}^{\beta}(x)-\left(x^{2},\frac{C_{1}^{\beta}}{\|C_{1}^{\beta}\|_{\pi_{\lambda,\beta}}^{2}}\right)_{\pi_{\lambda,\beta}}C_{1}^{\beta}(x),\\
 & =x^{2}-A(\beta,\lambda)x-m_{\lambda,\beta}(2)+A(\beta,\lambda)m_{\lambda,\beta}(1),
\end{align*}
where 
\[
A(\beta,\lambda)=\frac{m_{\lambda,\beta}(3)-m_{\lambda,\beta}(1)m_{\lambda,\beta}(2)}{m_{\lambda,\beta}(2)-(m_{\lambda,\beta}(1))^{2}}.
\]
When $\beta=1$, the measure $\pi_{\lambda,1}$ becomes the Poisson
measure $\pi_{\lambda}$ and the polynomials $C_{n}^{1}(x),$ $n\in\mathbb{N}_{0}$,
are the classical Charlier polynomials.
\begin{cor}
\label{cor:orthogonal-for-beta1}For $\beta\in(0,1]$ it holds
\[
\int_{\mathbb{R}^{2}}C_{1}^{\beta}(x_{1})C_{2}^{\beta}(x_{2})\,\mathrm{d}\pi_{\vec{\lambda},\beta}^{2}(x_{1},x_{2})=0
\]
if, and only if, $\beta=1.$
\end{cor}

\begin{proof}
When $\beta=1$, we have the well known orthogonal property of the
Charlier polynomials, that is,
\[
\int_{\mathbb{R}^{2}}C_{1}(x_{1})C_{2}(x_{2})\,\mathrm{d}\pi_{\vec{\lambda}}^{2}(x_{1},x_{2})=\int_{\mathbb{R}}C_{1}(x_{1})\,\mathrm{d}\pi_{\lambda}(x_{1})\int_{\mathbb{R}}C_{2}(x_{2})\,\mathrm{d}\pi_{\lambda}(x_{2})=0.
\]
On the other hand, for $\beta\in(0,1)$ we have
\begin{align}
 & \int_{\mathbb{R}^{2}}C_{1}^{\beta}(x_{1})C_{2}^{\beta}(x_{2})\,\mathrm{d}\pi_{\vec{\lambda},\beta}^{2}(x_{1},x_{2})\nonumber \\
 & =\int_{\mathbb{R}^{2}}\left(x_{1}-m_{\lambda_{1},\beta}(1)\right)\big(x_{2}^{2}-A(\beta,\lambda_{2})x_{2}-m_{\lambda_{2},\beta}(2)+A(\beta,\lambda_{2})m_{\lambda_{2},\beta}(1)\big)\,\mathrm{d}\pi_{\vec{\lambda},\beta}^{2}(x_{1},x_{2})\nonumber \\
 & =m_{\vec{\lambda},\beta}^{2}(1,2)-A(\beta,\lambda_{2})m_{\vec{\lambda},\beta}^{2}(1,1)-m_{\lambda_{1},\beta}(1)m_{\lambda_{2},\beta}(2)+A(\beta,\lambda_{2})m_{\lambda_{1},\beta}(1)m_{\lambda_{2},\beta}(1).\label{eq:moments-12}
\end{align}
Equation (3.3.3) defines a function $F(\beta,\lambda_{1},\lambda_{2})$
which is not equal from to zero for every $\beta\in(0,1)$, see Figure
3.1.
\end{proof}
\begin{figure}
\begin{centering}
\includegraphics[scale=0.4]{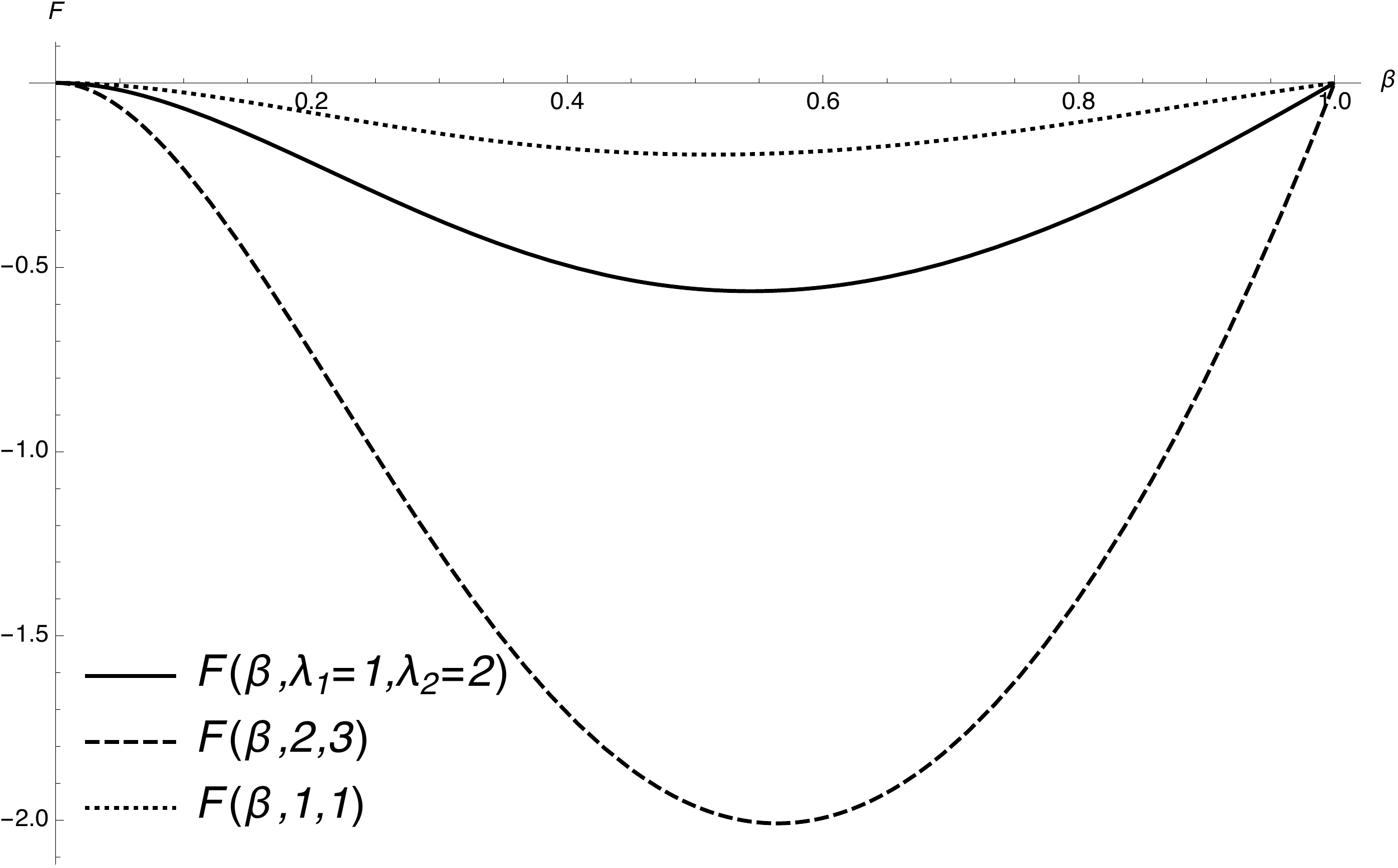}
\par\end{centering}
\caption{\label{fig:beta1}The graph of the function $F(\cdot,\lambda_{1},\lambda_{2})$
with $\vec{\lambda}=(1,1),(2,3),(1,2)$.}
\end{figure}

Having in mind the above results, this motivate us to introduce a
biorthogonal system of the fPm in higher dimension.

\section{Infinite Dimensional Fractional Poisson Measure}

\label{sec:Inf-Dim-fPm}After the above preparation, we are ready
to define the fPm in infinite dimensions. We define the fPm in the
linear space $\mathcal{D}'$ and then a more careful analysis shows
that fPm is indeed a probability measure on the configuration space
$\Gamma$ over $\mathbb{R}^{d}$.

\subsection{Fractional Poisson Measure on the Linear Space $\mathcal{D}'$}

\label{sec:fPm-on-D'}Let $\vec{\lambda}=(\lambda_{1},\dots,\lambda_{d})\in(\mathbb{R}_{+}^{*})^{d}$
and $z=(z_{1},\dots,z_{d})\in\mathbb{R}^{d}$ be given. The $d$-dimensional
Poisson measure has characteristic function given by
\begin{equation}
C_{\pi_{\lambda}^{d}}(z)=\int_{\mathbb{R}^{d}}\mathrm{e}^{\mathrm{i}(x,z)}\,\mathrm{d}\pi_{\vec{\lambda}}^{d}(x)=\exp\left(\sum_{k=1}^{d}\lambda_{k}(\mathrm{e}^{\mathrm{i}z_{k}}-1)\right).\label{eq:CF-Pm-fd}
\end{equation}

Let us consider a \emph{Radon measure} $\sigma$ on $(\mathbb{R}^{d},\mathcal{B}(\mathbb{R}^{d}))$,
that is, $\sigma(\Lambda)<\infty$ for every $\Lambda\in\mathcal{B}_{c}(\mathbb{R}^{d})$,
where $\mathcal{B}_{c}(\mathbb{R}^{d})$ the family of all $\mathcal{B}(\mathbb{R}^{d})$-measurable
sets with compact closure. Elements of $\mathcal{B}_{c}(\mathbb{R}^{d})$
are called finite volumes. Here, we assume $\sigma$ to be non-degenerate
(i.e., $\sigma(O)>0$ for all non-empty open sets $O\subset\mathbb{R}^{d}$)
and non-atomic (i.e., $\sigma(\left\{ x\right\} )=0$ for every $x\in\mathbb{R}^{d}$).
In addition, we always assume that $\sigma(\mathbb{R}^{d})=\infty$.
Let $\mathcal{D}:=\mathcal{D}(\mathbb{R}^{d})$ be the space of $C^{\infty}$-functions
with compact support in $\mathbb{R}^{d}$ and $\mathcal{D}':=\mathcal{D}'(\mathbb{R}^{d})$
be the dual of $\mathcal{D}$ with respect to the Hilbert space $L^{2}(\sigma):=L^{2}(\mathbb{R}^{d},\mathcal{B}(\mathbb{R}^{d}),\sigma)$.
In this way, we obtain the triple 
\begin{equation}
\mathcal{D}\subset L^{2}(\sigma)\subset\mathcal{D}'.\label{eq:triple-L2sigma}
\end{equation}
The infinite-dimensional generalization of the Poisson measure with
intensity measure $\sigma$, denoted by $\pi_{\sigma}$, is obtained
by generalizing the characteristic function \eqref{eq:CF-Pm-fd} to
\begin{equation}
C_{\pi_{\sigma}}(\varphi):=\int_{\mathcal{D}'}\mathrm{e}^{\mathrm{i}\langle w,\varphi\rangle}\,\mathrm{d}\pi_{\sigma}(w)=\exp\left(\int_{\mathbb{R}^{d}}(\mathrm{e}^{\mathrm{i}\varphi(x)}-1)\mathrm{\,d}\sigma(x)\right),\;\varphi\in\mathcal{D}.\label{eq:CF-Pm-inf}
\end{equation}
This is achieve through the Bochner-Minlos theorem (see e.g.~\cite{BK88})
by showing that $C_{\pi_{\sigma}}$ is the Fourier transform of a
measure on the distribution space $\mathcal{D}'$, see \cite{AKR97}
and references therein. Now, using the fact that the Mittag-Leffler
function is a natural generalization of the exponential function,
one conjectures that the characteristic functional 
\begin{equation}
C_{\pi_{\sigma}^{\beta}}(\varphi):=\int_{\mathcal{D}'}\mathrm{e}^{\mathrm{i}\langle w,\varphi\rangle}\,\mathrm{d}\pi_{\sigma}^{\beta}(w)=E_{\beta}\left(\int_{\mathbb{R}^{d}}(\mathrm{e}^{\mathrm{i}\varphi(x)}-1)\mathrm{\,d}\sigma(x)\right),\quad\varphi\in\mathcal{D},\label{eq:CF-fPm-inf}
\end{equation}
defines an infinite-dimensional version of the fPm, denoted by $\pi_{\sigma}^{\beta}$.
However, since the Mittag-Leffler function does not satisfy the semigroup
property of the exponential, it is not obvious that this is the Fourier
transform of a measure on $\mathcal{D}'$. Hence, we use the Bochner-Minlos
theorem to show that $C_{\pi_{\sigma}^{\beta}}$ is the Fourier transform
of a probability measure $\pi_{\sigma}^{\beta}$ on the distribution
space $\mathcal{D}'$.
\begin{thm}
\label{thm:characteristic-functional-of-fPm-on-D'}For each $0<\beta\leq1$
fixed, the functional $C_{\pi_{\sigma}^{\beta}}$ in Equation \eqref{eq:CF-fPm-inf}
is the characteristic functional on $\mathcal{D}$ of a probability
measure $\pi_{\sigma}^{\beta}$ on the distribution space $\mathcal{D}'$.
\end{thm}

\begin{proof}
Using the properties of the Mittag-Leffler function, the functional
$C_{\pi_{\sigma}^{\beta}}$ is continuous and $C_{\pi_{\sigma}^{\beta}}(0)=1$
follow directly. To show that the functional $C_{\pi_{\sigma}^{\beta}}$
is positive definite, we use the complete monotonicity property of
$E_{\beta}$, $0<\beta<1$, see \eqref{eq:monotonicity-Mittag}. For
any $\varphi_{i}\in\mathcal{D}$, $z_{i}\in\mathbb{C}$, $i=1,\dots,n$,
using Equation~\eqref{eq:CF-Pm-inf}, we obtain 
\begin{align*}
\sum_{k,j=1}^{n}C_{\pi_{\sigma}^{\beta}}(\varphi_{k}-\varphi_{j})z_{k}\bar{z}_{j} & =\int_{0}^{\infty}\sum_{k,j=1}^{n}\mathrm{e}^{\tau\int_{\mathbb{R}^{d}}(\mathrm{e}^{\mathrm{i}(\varphi_{k}-\varphi_{j})}-1)\,\mathrm{d}\sigma(x)}z_{k}\bar{z}_{j}\,\mathrm{d}\nu_{\beta}(\tau)\\
 & =\int_{0}^{\infty}\sum_{k,j=1}^{n}C_{\pi_{\tau\sigma}}(\varphi_{k}-\varphi_{j})z_{k}\bar{z}_{j}\,\mathrm{d}\nu_{\beta}(\tau).
\end{align*}
Using the definition of $C_{\pi_{\tau\sigma}}$, the integrand of
the last integral may be written as
\[
\sum_{k,j=1}^{n}C_{\pi_{\tau\sigma}}(\varphi_{k}-\varphi_{j})z_{k}\bar{z}_{j}=\int_{\mathcal{D}'}\left|\sum_{k=1}^{n}\mathrm{e}^{\mathrm{i}\langle w,\varphi_{k}\rangle}z_{k}\right|^{2}\mathrm{d}\pi_{\tau\sigma}(w)\geq0.
\]
This implies that $C_{\pi_{\sigma}^{\beta}}$ is positive-definite.
Thus by the Bochner-Minlos theorem, $C_{\pi_{\sigma}^{\beta}}$ is
the characteristic functional of a probability measure $\pi_{\sigma}^{\beta}$
on the measurable space $(\mathcal{D}',\mathcal{C}_{\sigma}(\mathcal{D}'))$.
\end{proof}
\begin{rem}
By the analytic property of the Mittag-Leffler function one may write
\eqref{eq:CF-fPm-inf} for any $\varphi\in\mathcal{D}$ such that
$\mathrm{supp}\,\varphi\subset\Lambda\in\mathcal{B}_{c}(\mathbb{R}^{d})$,
as
\begin{align*}
C_{\pi_{\sigma}^{\beta}}(\varphi) & =E_{\beta}\left(\int_{\mathbb{R}^{d}}(\mathrm{e}^{\mathrm{i}\varphi(x)}-1)\mathrm{\,d}\sigma(x)\right)=E_{\beta}\left(\int_{\Lambda}(\mathrm{e}^{\mathrm{i}\varphi(x)}-1)\mathrm{\,d}\sigma(x)\right)\\
 & =E_{\beta}\left(\int_{\Lambda}\mathrm{e}^{\mathrm{i}\varphi(x)}\mathrm{\,d}\sigma(x)-\sigma(\Lambda)\right)\\
 & =\sum_{n=0}^{\infty}\frac{E_{\beta}^{(n)}\left(-\sigma(\Lambda)\right)}{n!}\left(\int_{\Lambda}\mathrm{e}^{\mathrm{i}\varphi(x)}\,\mathrm{d}\sigma(x)\right)^{n}\\
 & =\sum_{n=0}^{\infty}\frac{E_{\beta}^{(n)}\left(-\sigma(\Lambda)\right)}{n!}\int_{\Lambda^{n}}\mathrm{e}^{\mathrm{i}(\varphi(x_{1})+\dots+\varphi(x_{n}))}\,\mathrm{d}\sigma^{\otimes n}(x_{1},\dots,x_{n}),
\end{align*}
where $\sigma^{\otimes n}=\sigma\otimes\dots\otimes\sigma$ is a measure
defined on the Cartesian space $(\mathbb{R}^{d})^{n}:=\mathbb{R}^{d}\times\dots\times\mathbb{R}^{d}.$
In the Poisson case, we have $\exp\left(-\sigma(\Lambda)\right)$
instead of $E_{\beta}^{(n)}\left(-\sigma(\Lambda)\right)$, for all
$n\in\mathbb{N}_{0}$, while the rest of the terms are the same. Hence,
the main difference between these measures ($\pi_{\sigma}^{\beta}$
and $\pi_{\sigma}$) is the different weight given in each $n$-particle
space. In Subsection \ref{sec:fPm-cspace} we show that, indeed, the
support of the measure $\pi_{\sigma}^{\beta}$ is a subset of $\mathcal{D}'$,
called the configuration space over $\mathbb{R}^{d}$.
\end{rem}

We may now generalize the result of Lemma \ref{lem:mixture-1d} to
the present infinite dimensional setting.
\begin{lem}
For $0<\beta\leq1$, the fPm $\pi_{\sigma}^{\beta}$ is an integral
(or mixture) of Poisson measure $\pi_{\sigma}$ with respect to the
probability measure $\nu_{\beta}$, i.e.,
\begin{equation}
\pi_{\sigma}^{\beta}=\int_{0}^{\infty}\pi_{\tau\sigma}\,\mathrm{d\nu_{\beta}(\tau)}.\label{eq:fPm-mixture-inf}
\end{equation}
\end{lem}

\begin{proof}
For $\beta=1$, the result is clear as in the proof of Lemma \ref{lem:mixture-1d}.
For $0<\beta<1$, we use the representation \eqref{eq:monotonicity-Mittag}
of the Mittag-Leffler function, the characteristic functional \eqref{eq:CF-fPm-inf}
of $\pi_{\sigma}^{\beta}$ can be rewritten as 
\[
C_{\pi_{\sigma}^{\beta}}(\varphi)=\int_{0}^{\infty}\exp\left(-\tau\int_{\mathbb{R}^{d}}(1-\mathrm{e}^{\mathrm{i}\varphi(x)})\mathrm{\,d}\sigma(x)\right)\mathrm{d}\nu_{\beta}(\tau)
\]
with the integrand being the characteristic function of the Poisson
measure $\pi_{\tau\sigma}$, $\tau>0$. This implies that the characteristic
functional \eqref{eq:CF-fPm-inf} coincides with the characteristic
functional of the measure $\int_{0}^{\infty}\pi_{\tau\sigma}\,\mathrm{d}\nu_{\beta}(\tau)$.
The result follows by the uniqueness of the characteristic functional.
\end{proof}
The fPm $\pi_{\sigma}^{\beta}$ is indeed a probability measure on
$(\mathcal{D}',\mathcal{C}_{\sigma}(\mathcal{D}'))$. In what follows,
we are going to find an appropriate support for $\pi_{\sigma}^{\beta}$.

\subsection{Configuration Space}

\label{subsec:configuration-space} Recall that $\mathcal{B}(\mathbb{R}^{d})$
denotes the Borel $\sigma$-algebra on $\mathbb{R}^{d}$ and $\mathcal{B}_{c}(\mathbb{R}^{d})$
the system of all sets in $\mathcal{B}(\mathbb{R}^{d})$ which are
bounded and have compact closure. Below we recall the configuration
space over $\mathbb{R}^{d}$ and related concepts, see \cite{AKR97,KK02}
for more details.
\begin{defn}
The \textit{infinite configuration space}\textit{\emph{\index{configuration space@\textit{configuration space}}}}
$\Gamma:=\Gamma_{\mathbb{R}^{d}}$ over $\mathbb{R}^{d}$ is defined
as the set of all locally finite subsets from $\mathbb{R}^{d}$, that
is, 
\[
\Gamma:=\big\{\gamma\subset\mathbb{R}^{d}:|\gamma\cap\Lambda|<\infty\thinspace\text{for every }\thinspace\Lambda\in\mathcal{B}_{c}(\mathbb{R}^{d})\big\},
\]
where $|B|$ denotes the cardinality of the set $B$. The\index{sigma-algebra@$\sigma$-algebra}
elements of the space $\Gamma$ are called \emph{configurations}.
\end{defn}

Let $C_{0}(\mathbb{R}^{d})$ denote the class of all real-valued continuous
functions on $\mathbb{R}^{d}$ with compact support and $\mathcal{M}^{+}:=\mathcal{M}^{+}(\mathbb{R}^{d})$\nomenclature[positive measures on Rd]{\$\textbackslash\{\}mathcal\{M\}\^\{\}\{+\}:=\textbackslash\{\}mathcal\{M\}\^\{\}\{+\}(\textbackslash\{\}mathbb\{R\}\^\{\}\{d\})\$}{positive measures on \$\textbackslash\{\}mathbb\{R\}\^\{\}\{d\}\$}
(resp.~$\mathcal{M}_{\mathbb{N}_{0}}^{+}:=\mathcal{M}_{\mathbb{N}_{0}}^{+}(\mathbb{R}^{d})$)
denote the space of all positive (resp.~positive integer-valued)
Radon measures on $\mathcal{B}(\mathbb{R}^{d})$.
\begin{defn}
Each configuration $\gamma\in\Gamma$ can be identified with a non-negative
integer-valued Radon measure as follows
\[
\Gamma\ni\gamma\mapsto\sum_{x\in\gamma}\delta_{x}\in\mathcal{M}_{\mathbb{N}_{0}}^{+}\subset\mathcal{M}^{+},
\]
where $\delta_{x}$ is the Dirac measure at $x\in\mathbb{R}^{d}$
and $\sum_{x\in\emptyset}\delta_{x}:=0$ (zero measure). The space
$\Gamma$ can be endowed with the topology induced by the vague topology
on $\mathcal{M}^{+}$, i.e., the weakest topology on $\Gamma$ with
respect to which all mappings

\[
\Gamma\ni\gamma\mapsto\langle\gamma,f\rangle:=\langle f\rangle_{\gamma}:=\int_{\mathbb{R}^{d}}f(x)\,\mathrm{d}\gamma(x)=\sum_{x\in\gamma}f(x)\in\mathbb{R}
\]
are continuous for any $f\in C_{0}(\mathbb{R}^{d})$. 
\end{defn}

\begin{defn}
Let $\mathcal{B}(\Gamma)$ be the Borel $\sigma$-algebra corresponding
to the vague topology on $\Gamma$.
\begin{enumerate}
\item The $\sigma$-algebra $\mathcal{B}(\Gamma)$ is generated by the sets
of the form 
\begin{equation}
C_{\Lambda,n}=\left\{ \gamma\in\Gamma\,|\,|\gamma\cap\Lambda|=n\right\} \!,\label{eq:cylinder-sets}
\end{equation}
where $\Lambda\in\mathcal{B}_{c}(\mathbb{R}^{d})$, $n\in\mathbb{N}_{0}$,
and the set $C_{\Lambda,n}$ is a Borel set of $\Gamma$, that is,
$C_{\Lambda,n}\in\mathcal{B}(\Gamma)$. Sets of the form \eqref{eq:cylinder-sets}
are called \emph{cylinder sets}\index{sets!cylinder}.
\item For any $B\subset\mathbb{R}^{d}$, we introduce a function $N_{B}:\Gamma\rightarrow\mathbb{N}_{0}$
such that 
\begin{align*}
N_{B}(\gamma): & =|\gamma\cap B|,\quad\gamma\in\Gamma.
\end{align*}
Then $\mathcal{B}(\Gamma)$ is the minimal $\sigma$-algebra with
which all functions $N_{\Lambda}$, $\Lambda\in\mathcal{B}_{c}(\mathbb{R}^{d})$,
are measurable.
\end{enumerate}
\end{defn}

\begin{defn}
Let $Y\in\mathcal{B}(\mathbb{R}^{d})$ be given. The space\emph{ }of
configurations contained in\emph{ $Y$} is denoted by \emph{$\Gamma_{Y}$,
}i.e., 
\[
\Gamma_{Y}:=\left\{ \gamma\in\Gamma\mid|\gamma\cap(\mathbb{R}^{d}\setminus Y)|=0\right\} .
\]
The $\sigma$-algebra $\mathcal{B}(\Gamma_{Y})$ may be introduced
in a similar way 
\[
\mathcal{B}(\Gamma_{Y}):=\sigma\left(\left\{ N_{\Lambda}{\upharpoonright}_{\Gamma_{Y}}\mid\Lambda\in\mathcal{B}_{c}(\mathbb{R}^{d})\right\} \right),
\]
where $N_{\Lambda}{\upharpoonright}_{\Gamma_{Y}}$ denotes the restriction
of the mapping $N_{\Lambda}$ to $\Gamma_{Y}$. 
\end{defn}

\begin{defn}
Let $Y\in\mathcal{B}(\mathbb{R}^{d})$ be given. The \textit{space
of n-point configurations} $\Gamma_{Y}^{(n)}$ over a set $Y$ is
the subset of $\Gamma_{Y}$ defined by
\[
\Gamma_{Y}^{(n)}:=\{\gamma\in\Gamma_{Y}\mid|\gamma|=n\},n\in\mathbb{N},\quad\Gamma_{Y}^{(0)}:=\{\emptyset\}.
\]
\end{defn}

A topological structure may be introduced on $\Gamma_{Y}^{(n)}$,
$n\in\mathbb{N}$, through a natural surjective mapping from 
\[
\widetilde{Y^{n}}=\{(x_{1},\dotsc,x_{n})\mid x_{k}\in Y,x_{k}\neq x_{j}\thinspace\text{if}\thinspace k\neq j\}
\]
onto $\Gamma_{Y}^{(n)}$ defined by 
\[
\mathrm{sym}_{Y}^{n}:\widetilde{Y^{n}}\longrightarrow\Gamma_{Y}^{(n)}
\]
\[
(x_{1},\dotsc,x_{n})\mapsto\{x_{1},\dotsc,x_{n}\}.
\]
Indeed, using the mapping $\mathrm{sym}_{\Lambda}^{n},$ one constructs
a bijective mapping between $\Gamma_{Y}^{(n)}$ and the symmetrization
$\widetilde{Y^{n}}/S_{n}$ of $\widetilde{Y^{n}},$ where $S_{n}$
is the permutation group over $\{1,\dotsc,n\}.$ In this way, $\mathrm{sym}_{Y}^{n}$
induces a metric on $\Gamma_{Y}^{(n)}$. A set $U\subset\Gamma_{Y}^{(n)}$
is open in this topology if, and only if, the inverse image $(\mathrm{sym}_{Y}^{n})^{-1}(U)$
is open in $\widetilde{Y^{n}}.$ We denote by $\mathcal{B}(\Gamma_{Y}^{(n)})$
the corresponding Borel $\sigma$-algebra and $\mathcal{T}_{Y}^{(n)}$
the associated topology on $\Gamma_{Y}^{(n)}$.

For $\Lambda\in\mathcal{B}_{c}(\mathbb{R}^{d}),$ each space $\Gamma_{\Lambda}$
can be described by the disjoint union 
\[
\Gamma_{\Lambda}=\bigsqcup_{n=0}^{\infty}\Gamma_{\Lambda}^{(n)}.
\]
In particular, this representation provides an equivalent description
of the $\sigma$-algebra $\mathcal{B}(\Gamma_{\Lambda})$ as the $\sigma$-algebra
of the disjoint union of the $\sigma$-algebras $\mathcal{B}(\Gamma_{\Lambda}^{(n)})\thinspace,n\in\mathbb{N}_{0}.$
The corresponding topology is denoted by $\mathcal{T}_{\Lambda}$
such that $(\Gamma_{\Lambda},\mathcal{T}_{\Lambda})$ is a topological
space for each $\Lambda\in\mathcal{B}_{c}(\mathbb{R}^{d})$.

For each $\Lambda\in\mathcal{B}_{c}(\mathbb{R}^{d})$ and any pair
$\Lambda_{1},\Lambda_{2}\in\mathcal{B}_{c}(\mathbb{R}^{d})$ such
that $\Lambda_{1}\subset\Lambda_{2},$ let us consider the natural
measurable projections 
\begin{align}
p_{\Lambda}:\Gamma\longrightarrow\Gamma_{\Lambda} &  & p_{\Lambda_{1},\Lambda_{2}}:\Gamma_{\Lambda_{2}}\longrightarrow\Gamma_{\Lambda_{1}}\label{eq:projections}\\
\gamma\mapsto\gamma\cap\Lambda &  & \gamma\mapsto\gamma\cap\Lambda_{1}.\nonumber 
\end{align}
We use now the concept of the projective limit in order to show
that the measurable space $(\Gamma,\mathcal{B}(\Gamma))$ coincides
with the projective limit. More precisely, we have the following theorem.
\begin{thm}
The family $\left\{ (\Gamma_{\Lambda},\mathcal{B}(\Gamma_{\Lambda})),p_{\Lambda_{1},\Lambda_{2}},\mathcal{B}_{c}(\mathbb{R}^{d})\right\} $
is a projective system of measurable spaces with ordered index set
$(\mathcal{B}_{c}(\mathbb{R}^{d}),\subset)$ and the measurable space
$(\Gamma,\mathcal{B}(\Gamma))$ is (up to an isomorphism) the projective
limit together with the family of maps $p_{\Lambda}:\Gamma\longrightarrow\Gamma_{\Lambda}$
for any $\Lambda\in\mathcal{B}_{c}(\mathbb{R}^{d})$. In addition,
we have the following commutative diagram.

\[
\xyC{6pc}\xyR{4pc}\xymatrix{ & \Gamma'\ar@{->}[d]_{I}\ar[ddl]_{P_{\Lambda_{2}}}\ar[ddr]^{P_{\Lambda_{1}}}\\
 & \Gamma\ar[dl]^{p_{\Lambda_{2}}}\ar[dr]_{p_{\Lambda_{1}}}\\
\Gamma_{\Lambda_{2}}\ar[rr]_{p_{\Lambda_{2},\Lambda_{1}}} &  & \Gamma_{\Lambda_{1}}
}
\]

\end{thm}

\begin{proof}
It is clear from the construction above that the maps $p_{\Lambda_{1},\Lambda_{2}}$,
$\Lambda_{1},\Lambda_{2}\in\mathcal{B}_{c}(\mathbb{R}^{d})$ are measurable
and satisfies 
\[
p_{\Lambda_{1},\Lambda_{2}}\circ p_{\Lambda_{2},\Lambda_{3}}=p_{\Lambda_{1},\Lambda_{3}},\quad\Lambda_{1}\subset\Lambda_{2}\subset\Lambda_{3}\:\:\mathrm{in\:\:}\mathcal{B}_{c}(\mathbb{R}^{d}).
\]
As a result, $\left\{ (\Gamma_{\Lambda},\mathcal{B}(\Gamma_{\Lambda})),p_{\Lambda_{1},\Lambda_{2}},\mathcal{B}_{c}(\mathbb{R}^{d})\right\} $
is a projective system. On the other hand, it is easy to see from
\eqref{eq:projections} that the following relation 
\[
p_{\Lambda_{1}}=p_{\Lambda_{1},\Lambda_{2}}\circ p_{\Lambda_{2}},\quad\Lambda_{1}\subset\Lambda_{2}\:\:\mathrm{in\:\:}\mathcal{B}_{c}(\mathbb{R}^{d})
\]
holds. By definition of $\mathcal{B}(\Gamma)$, the family of maps
$p_{\Lambda}$, $\Lambda\in\mathcal{B}_{c}(\mathbb{R}^{d})$ satisfy
the conditions in a projective limit of measurable spaces. This concludes
the proof.
\end{proof}

\subsection{Fractional Poisson Measure on $\Gamma$}

\label{sec:fPm-cspace}Recall from Section \ref{sec:fPm-on-D'} the
measure $\sigma$ on the underlying measurable space $(\mathbb{R}^{d},\mathcal{B}(\mathbb{R}^{d}))$
and the product measure $\sigma^{\otimes n}$ on $((\mathbb{R}^{d})^{n},\mathcal{B}((\mathbb{R}^{d})^{n}))$,
for each $n\in\mathbb{N}$. Then $\sigma^{\otimes n}((\mathbb{R}^{d})^{n}\setminus\widetilde{(\mathbb{R}^{d})^{n}})=0$,
since $\sigma$ is non-atomic. It follows that $\sigma^{\otimes n}(B{}^{n}\setminus\widetilde{B^{n}})=0$,
for every $B\in\mathcal{B}(\mathbb{R}^{d})$. For each $\Lambda\in\mathcal{B}_{c}(\mathbb{R}^{d})$,
let us consider the restriction of $\sigma^{\otimes n}$ to $(\widetilde{\Lambda^{n}},\mathcal{B}(\widetilde{\Lambda^{n}}))$,
which is a finite measure, and then define the image measure $\sigma_{\Lambda}^{(n)}$
on $(\Gamma_{\Lambda}^{(n)},\mathcal{B}(\Gamma_{\Lambda}^{(n)}))$
under the mapping $\mathrm{sym}_{\Lambda}^{n}$ by
\[
\sigma_{\Lambda}^{(n)}:=\sigma^{\otimes n}\circ(\mathrm{sym}_{\Lambda}^{n})^{-1}.
\]
For $n=0$, we set $\sigma_{\Lambda}^{(0)}(\left\{ \emptyset\right\} ):=1$.
Now, for each $0<\beta<1$, one may define a probability measure $\pi_{\sigma,\Lambda}^{\beta}$
on $(\Gamma_{\Lambda},\mathcal{B}(\Gamma_{\Lambda}))$ by
\begin{equation}
\pi_{\sigma,\Lambda}^{\beta}:=\sum_{n=0}^{\infty}\frac{E_{\beta}^{(n)}(-\sigma(\Lambda))}{n!}\sigma_{\Lambda}^{(n)}.\label{eq:fPm-on-cs}
\end{equation}
Note that $E_{\beta}^{(n)}(-\sigma(\Lambda))\geq0$, $n\in\mathbb{N}_{0}$
due to \eqref{eq:monotonicity-Mittag}. In addition, we have 
\begin{align*}
\pi_{\sigma,\Lambda}^{\beta}(\Gamma_{\Lambda}) & =\pi_{\sigma,\Lambda}^{\beta}\left(\bigsqcup_{n=0}^{\infty}\Gamma_{\Lambda}^{(n)}\right)=\sum_{n=0}^{\infty}\frac{E_{\beta}^{(n)}(-\sigma(\Lambda))}{n!}\sigma_{\Lambda}^{(n)}(\Gamma_{\Lambda}^{(n)})\\
 & =\sum_{n=0}^{\infty}\frac{E_{\beta}^{(n)}(-\sigma(\Lambda))}{n!}\sigma(\Lambda)^{n}=E_{\beta}(0)=1.
\end{align*}
The family $\{\pi_{\sigma,\Lambda}^{\beta}\mid\Lambda\in\mathcal{B}_{c}(\mathbb{R}^{d})\}$
of probability measures yields a probability measure on $(\Gamma,\mathcal{B}(\Gamma))$.
In fact, this family is consistent in the sense that the measure $\pi_{\sigma,\Lambda_{1}}^{\beta}$
is the image measure of $\pi_{\sigma,\Lambda_{2}}^{\beta}$ under
$p_{\Lambda_{1},\Lambda_{2}}$, that is, 
\[
\pi_{\sigma,\Lambda_{1}}^{\beta}=\pi_{\sigma,\Lambda_{2}}^{\beta}\circ p_{\Lambda_{1},\Lambda_{2}}^{-1},\quad\forall\Lambda_{1},\Lambda_{2}\in\mathcal{B}_{c}(\mathbb{R}^{d}),\Lambda_{1}\subset\Lambda_{2}.
\]
By the Kolmogorov extension theorem on configuration space (see Appendix~\ref{sec:Kolmogorov-extension-theorem-on-config-space})
the family $\big\{\pi_{\sigma,\Lambda}^{\beta}\mid\Lambda\in\mathcal{B}_{c}(\mathbb{R}^{d})\big\}$
determines uniquely a measure $\pi_{\sigma}^{\beta}$ on $(\Gamma,\mathcal{B}(\Gamma))$
such that
\[
\pi_{\sigma,\Lambda}^{\beta}=\pi_{\sigma}^{\beta}\circ p_{\Lambda}^{-1},\quad\forall\Lambda\in\mathcal{B}_{c}(\mathbb{R}^{d}).
\]
Actually, we don't need the whole family of local sets $\mathcal{B}_{c}(\mathbb{R}^{d})$
instead, a sub-family $\mathcal{J}_{\mathbb{R}^{d}}$ with an abstract
concept of local sets, see (I1)--(I3) in Appendix \ref{sec:Kolmogorov-extension-theorem-on-config-space}.

Let us now compute the characteristic functional of the measure $\pi_{\sigma}^{\beta}$.
Given $\varphi\in\mathcal{D}$, we have $\mathrm{supp\,}\varphi\subset\Lambda$
for some $\Lambda\in\mathcal{B}_{c}(\mathbb{R}^{d})$, such that
\[
\langle\gamma,\varphi\rangle=\langle p_{\Lambda}(\gamma),\varphi\rangle,\hspace{0.1in}\forall~\gamma\in\Gamma.
\]
Thus, 
\[
\int_{\Gamma}\mathrm{e}^{\mathrm{i}\langle\gamma,\varphi\rangle}\mathrm{d}\pi_{\sigma}^{\beta}(\gamma)=\int_{\Gamma_{\Lambda}}\mathrm{e}^{\mathrm{i}\langle\gamma,\varphi\rangle}\mathrm{d}\pi_{\sigma,\Lambda}^{\beta}(\gamma)
\]
and the infinite divisibility \eqref{eq:fPm-on-cs} of the measure
$\pi_{\sigma,\Lambda}^{\beta}$ yields for the right-hand side of
the above equality
\[
\sum_{n=0}^{\infty}\frac{E_{\beta}^{(n)}(-\sigma(\Lambda))}{n!}\int_{\Lambda^{n}}\mathrm{e}^{\mathrm{i}(\varphi(x_{1})+\dots+\varphi(x_{n}))}\,\mathrm{d}\sigma^{\otimes n}(x_{1},\dots,x_{n})=\sum_{n=0}^{\infty}\frac{E_{\beta}^{(n)}(-\sigma(\Lambda))}{n!}\left(\int_{\Lambda}\mathrm{e}^{\mathrm{i}\varphi(x)}\,\mathrm{d}\sigma(x)\right)^{n}
\]
which corresponds to the Taylor expansion of the function 
\[
E_{\beta}\left(\int_{\Lambda}(\mathrm{e}^{\mathrm{i}\varphi(x)}-1)\,\mathrm{d}\sigma(x)\right)=E_{\beta}\left(\int_{\mathbb{R}^{d}}(\mathrm{e}^{\mathrm{i}\varphi(x)}-1)\,\mathrm{d}\sigma(x)\right).
\]
Hence, for any $\varphi\in\mathcal{D}$, we obtain 
\begin{equation}
\int_{\Gamma}\mathrm{e}^{\mathrm{i}\langle\gamma,\varphi\rangle}\,\mathrm{d}\pi_{\sigma}^{\beta}(\gamma)=E_{\beta}\left(\int_{\mathbb{R}^{d}}(\mathrm{e}^{\mathrm{i}\varphi(x)}-1)\,\mathrm{d}\sigma(x)\right).\label{eq:characteristic-function-Gamma}
\end{equation}

\begin{rem}
\label{rem:fPm-Gamma-Dprime}
\begin{enumerate}
\item The characteristic functional of the measure $\pi_{\sigma}^{\beta}$
given in \eqref{eq:characteristic-function-Gamma} coincides with
the characteristic functional \eqref{eq:CF-fPm-inf} of the measure
$\pi_{\sigma}^{\beta}$ on the distribution space $\mathcal{D}'$.
The functional \eqref{eq:characteristic-function-Gamma} shows that
the measure $\pi_{\sigma}^{\beta}$ is supported on generalized functions
of the form $\sum_{x\in\gamma}\delta_{x}\in\mathcal{D}',$ $\gamma\in\Gamma$.
\item Note that $\Gamma\subset\mathcal{D}'$ but in contrast to $\Gamma,$
$\mathcal{D}'$ is a linear space. Since $\pi_{\sigma}^{\beta}(\Gamma)=1$,
the measure space $(\mathcal{D}',\mathcal{C}_{\sigma}(\mathcal{D}'),\pi_{\sigma}^{\beta})$
can, in this way, be regarded as a linear extension of the fractional
Poisson space $(\Gamma,\mathcal{B}(\Gamma),\pi_{\sigma}^{\beta}).$
\end{enumerate}
\end{rem}

\section{Generalized Appell System}

\label{sec:Appell-System}In this section we introduce the generalized
Appell system associated with the fPm $\pi_{\sigma}^{\beta}$. First
we consider the analytic continuation of the characteristic functional
$C_{\pi_{\sigma}^{\beta}}$ to $\mathcal{D}_{\mathbb{C}}:=\mathcal{D}\oplus\mathrm{i}\mathcal{D}$.
By definition, an element $\varphi\in\mathcal{D}_{\mathbb{C}}$ decomposes
into $\varphi=\varphi_{1}+\mathrm{i}\varphi_{2}$, $\varphi_{1},\varphi_{2}\in\mathcal{D}$.
Hence, computing $C_{\pi_{\sigma}^{\beta}}(-\mathrm{i}\varphi)$,
$\varphi\in\mathcal{D}$, yields the Laplace transform of the measure
$\pi_{\sigma}^{\beta}$, that is, 
\[
l_{\pi_{\sigma}^{\beta}}(\varphi):=C_{\pi_{\sigma}^{\beta}}(-\mathrm{i}\varphi)=E_{\beta}\left(\int_{\mathbb{R}^{d}}(\mathrm{e}^{\varphi(x)}-1)\,\mathrm{d}\sigma(x)\right).
\]
In particular, choosing $\beta=1$ we obtain the Laplace transform
of the classical Poisson measure $\pi_{\sigma}:=\pi_{\sigma}^{1}$
with intensity $\sigma$ on the configuration space $\Gamma$. For
more details, we refer to \cite{GV68,KMM78,I88,IK88,AKR97a} and reference
therein.

The following two properties are satisfied by the fPm $\pi_{\sigma}^{\beta}$,
$0<\beta\leq1$.
\begin{description}
\item [{(A1)}] The measure $\pi_{\sigma}^{\beta}$ has an analytic Laplace
transform in a neighborhood of zero, that is, the mapping
\[
\mathcal{D_{\mathbb{C}}}\ni\varphi\mapsto l_{\pi_{\sigma}^{\beta}}(\varphi)=\int_{\mathcal{D}'}\mathrm{e}^{\langle w,\varphi\rangle}\,\mathrm{d}\pi_{\sigma}^{\beta}(w)=E_{\beta}\left(\int_{\mathbb{R}^{d}}(\mathrm{e}^{\varphi(x)}-1)\,\mathrm{d}\sigma(x)\right)\in\mathbb{C}
\]
is holomorphic in a neighborhood $\mathcal{U}\subset\mathcal{D}_{\mathbb{C}}$
of zero.
\item [{(A2)}] For any nonempty open subset $\mathcal{U}\subset\mathcal{D}'$
it should hold that $\pi_{\sigma}^{\beta}(\mathcal{U})>0$.
\end{description}
The assumption (A1) guarantees the existence of the moments of all
order of the measure $\pi_{\sigma}^{\beta}$ while (A2) guarantees
the embedding of the test function space on $L^{2}(\pi_{\sigma}^{\beta})$,
see e.g., Section~3 in \cite{KaKo99}. In addition, the Laplace transform
$l_{\pi_{\sigma}^{\beta}}(\varphi)$ of the measure $\pi_{\sigma}^{\beta}$
has the decomposition in terms of the moment kernels $M_{n}^{\sigma,\beta}$
(by the kernel theorem) given by
\begin{equation}
l_{\pi_{\sigma}^{\beta}}(\varphi)=\sum_{n=0}^{\infty}\frac{1}{n!}\langle M_{n}^{\sigma,\beta},\varphi^{\otimes n}\rangle,\qquad\varphi\in\mathcal{D}_{\mathbb{C}},\,M_{n}^{\sigma,\beta}\in\big(\mathcal{D}_{\mathbb{C}}^{^{\hat{\otimes}n}}\big)'.\label{eq:Laplace-transform-inf-dim}
\end{equation}

\subsection{Generalized Appell Polynomials}

\label{sec:Generalized-Appell-Polynomials}In this subsection we follow
\cite{KSWY95} to introduce the system of Appell polynomials associated
with the fPm $\pi_{\sigma}^{\beta}$. Let us consider the triple \eqref{eq:triple-L2sigma}
such that 
\begin{equation}
\mathcal{D}\subset\mathcal{N}\subset L^{2}(\sigma)\subset\mathcal{N}'\subset\mathcal{D}'\label{eq:triple-L2sigma-with-N}
\end{equation}
as described in Section~\ref{sec:Nuclear-spaces}. Also, the chain
\eqref{eq:triple-L2sigma-with-N} holds for the tensor product of
these spaces.

Then we introduce the normalized exponential $\mathrm{e}_{\pi_{\sigma}^{\beta}}(\varphi;z)$
by 
\begin{equation}
\mathrm{e}_{\pi_{\sigma}^{\beta}}(\varphi;w)=\frac{\mathrm{e}^{\langle w,\varphi\rangle}}{l_{\pi_{\sigma}^{\beta}}(\varphi)},\quad w\in\mathcal{D}'_{\mathbb{C}},\,\varphi\in\mathcal{D}_{\mathbb{C}}.\label{eq:normalized-expo-inf-dim}
\end{equation}
Since $l_{\pi_{\sigma}^{\beta}}(0)=1$ and $l_{\pi_{\sigma}^{\beta}}$
is holomorphic, there exists a neighborhood $\mathcal{U}_{0}\subset\mathcal{D}_{\mathbb{C}}$
of zero, such that $l_{\pi_{\sigma}^{\beta}}(\varphi)\neq0$ for all
$\varphi\in\mathcal{U}_{0}$. For $\varphi\in\mathcal{U}_{0}$, the
normalized exponential $\mathrm{e}_{\pi_{\sigma}^{\beta}}(\varphi;z)$
can be expanded in a power series and then we use the polarization
identity in order to apply the kernel theorem to obtain
\begin{equation}
\mathrm{e}_{\pi_{\sigma}^{\beta}}(\varphi;w)=\sum_{n=0}^{\infty}\frac{1}{n!}\langle P_{n}^{\sigma,\beta}(w),\varphi^{\otimes n}\rangle,\quad w\in\mathcal{D}'_{\mathbb{C}},\;\varphi\in\mathcal{U}_{0},\label{eq:Wick-Appell-inf-dim}
\end{equation}
for suitable $P_{n}^{\sigma,\beta}(w)\in(\mathcal{D}_{\mathbb{C}}^{\hat{\otimes}n})'$.
The family 
\begin{equation}
\mathbb{P}^{\sigma,\beta}=\left\{ \langle P_{n}^{\sigma,\beta}(\cdot),\varphi^{(n)}\rangle\mid\varphi^{(n)}\in\mathcal{D}_{\mathbb{C}}^{\hat{\otimes}n},n\in\mathbb{N}_{0}\right\} \label{eq:P-system}
\end{equation}
is called the \emph{Appell system }associated to the fPm\emph{ $\pi_{\sigma}^{\beta}$}.
Let us now consider the transformation $\alpha:\mathcal{D}_{\mathbb{C}}\longrightarrow\mathcal{D}_{\mathbb{C}}$
defined on a neighborhood $\mathcal{U}_{\alpha}\subset\mathcal{D}_{\mathbb{C}}$
of zero, by 
\[
\alpha(\varphi)(x)=\log(1+\varphi(x)),\quad\varphi\in\mathcal{U}_{\alpha},\:x\in\mathbb{R}^{d}.
\]
Note that for $\varphi=0\in\mathcal{D_{\mathbb{C}}}$, we have $\alpha(\varphi)=0$.
Also, $\mathcal{U}_{\alpha}$ is chosen in such a way that $\alpha$
is invertible and holomorphic on $\mathcal{U}_{\alpha}$. Then $\alpha$
can be expanded as 
\begin{equation}
\alpha(\varphi)=\sum_{n=1}^{\infty}\frac{1}{n!}\widehat{\mathrm{d}^{n}\alpha(0)}(\varphi)=\sum_{n=1}^{\infty}\frac{(-1)^{n+1}\varphi^{n}}{n},\label{eq:alpha-Taylor}
\end{equation}
where 
\[
\widehat{\mathrm{d}^{n}\alpha(0)}(\varphi)=\frac{\partial^{n}}{\mathrm{\partial}t_{1}\dots\mathrm{\partial}t_{n}}\alpha(t_{1}\varphi+\dots+t_{n}\varphi)|_{t_{1}=\dots=t_{n}=0}
\]
for all $n\in\mathbb{N}$. For the inverse function $g_{\alpha}$
of $\alpha$, we have 
\[
(g_{\alpha}\varphi)(x)=\mathrm{e}^{\varphi(x)}-1,\qquad\varphi\in\mathcal{V}_{\alpha}\subset\mathcal{D}_{\mathbb{C}},\:x\in\mathbb{R}^{d}
\]
for some neighborhood $\mathcal{V}_{\alpha}$ of zero in $\mathcal{D}_{\mathbb{C}}$.
A similar procedure as before yields the decomposition 
\begin{equation}
g_{\alpha}(\varphi)=\sum_{n=1}^{\infty}\frac{1}{n!}\widehat{\mathrm{d}^{n}g_{\alpha}(0)}(\varphi)=\sum_{n=1}^{\infty}\frac{\varphi^{n}}{n!}.\label{eq:g-decomposition}
\end{equation}
Now using the function $\alpha$, we introduce the modified normalized
exponential $\mathrm{e}_{\pi_{\sigma}^{\beta}}(\alpha(\varphi);x)$
as 
\begin{equation}
\mathrm{e}_{\pi_{\sigma}^{\beta}}(\alpha(\varphi);w):=\frac{\exp(\langle w,\alpha(\varphi)\rangle)}{l_{\pi_{\sigma}^{\beta}}(\alpha(\varphi))}=\frac{\exp(\langle w,\log(1+\varphi)\rangle)}{E_{\beta}\left(\int_{\mathbb{R}^{d}}\varphi(x)\mathrm{\,d}\sigma(x)\right)}=\frac{\exp(\langle w,\log(1+\varphi)\rangle)}{E_{\beta}\left(\langle\varphi\rangle_{\sigma}\right)}\label{eq:normalized-expo-inf-dim-alpha}
\end{equation}
for $\varphi\in\mathcal{U}'_{\alpha}\subset\mathcal{U_{\alpha}}$,
$w\in\mathcal{D}_{\mathbb{C}}^{'}$. Since $l_{\pi_{\sigma}^{\beta}}$
is holomorphic on a neighborhood of zero, for each fixed $w\in\mathcal{D}'_{\mathbb{C}}$,
$\mathrm{e}_{\pi_{\sigma}^{\beta}}(\alpha(\cdot);w)$ is a holomorphic
function on some neighborhood $\mathcal{U}'_{\alpha}\subset\mathcal{U}_{\alpha}$
of zero. Then we have the map $\mathcal{D}_{\mathbb{C}}\ni\varphi\mapsto\mathrm{e}_{\pi_{\sigma}^{\beta}}(\alpha(\varphi);w)$
which admits a power series
\begin{equation}
\mathrm{e}_{\pi_{\sigma}^{\beta}}(\alpha(\varphi);w)=\sum_{n=0}^{\infty}\frac{1}{n!}\langle C_{n}^{\sigma,\beta}(w),\varphi^{\otimes n}\rangle,\quad\varphi\in\mathcal{U}'_{\alpha},\;w\in\mathcal{D}_{\mathbb{C}}^{'},\label{eq:normalized-exp-fPm-inf-dim}
\end{equation}
where the kernels $C_{n}^{\sigma,\beta}:\mathcal{D}'_{\mathbb{C}}\rightarrow(\mathcal{D}_{\mathbb{C}}^{^{\hat{\otimes}n}})'$,
$n\in\mathbb{N}$, $C_{0}^{\sigma,\beta}=1$. By Equation \eqref{eq:normalized-exp-fPm-inf-dim},
it follows that for any $\varphi^{(n)}\in\mathcal{D}_{\mathbb{C}}^{\hat{\otimes}n}$,
$n\in\mathbb{N}_{0}$, the function 
\[
\mathcal{D}'_{\mathbb{C}}\ni w\mapsto\langle C_{n}^{\sigma,\beta}(w),\varphi^{(n)}\rangle
\]
is a polynomial of order $n$ on $\mathcal{D}'_{\mathbb{C}}$.
\begin{defn}
The family 
\[
\mathbb{P}^{\sigma,\beta,\alpha}=\left\{ \langle C_{n}^{\sigma,\beta}(\cdot),\varphi^{(n)}\rangle\mid\varphi^{(n)}\in\mathcal{D}_{\mathbb{C}}^{\hat{\otimes}n},n\in\mathbb{N}_{0}\right\} 
\]
is called the \emph{generalized Appell system }associated to the fPm\emph{
$\pi_{\sigma}^{\beta}$ }or the\emph{ $\mathbb{P}^{\sigma,\beta,\alpha}$-system.}

In the following proposition we collect some properties of the kernels
$C_{n}^{\sigma,\beta}(\cdot)$ which appeared in \cite{KdSS98} but
specific to the measure \emph{$\pi_{\sigma}^{\beta}$}.
\end{defn}

\begin{prop}
\label{prop:generalized-appell-polynomials-inf-dim}For $z,w\in\mathcal{D}_{\mathbb{C}}'$,
$n\in\mathbb{N}_{0}$, the following properties hold
\begin{description}
\item [{(P1)}] $C_{n}^{\sigma,\beta}(w)=\sum_{m=0}^{n}\mathbf{s}(n,m)^{*}P_{m}^{\sigma,\beta}(w)$,
where $\mathbf{s}(n,m)$ is the Stirling operator of the first kind
defined in \eqref{eq:Stirling-first-D-property} in Appendix~\ref{sec:Stirling-Operators}.
\item [{(P2)}] $w^{\otimes n}=\sum_{k=0}^{n}\sum_{m=0}^{k}\binom{n}{k}\mathbf{S}(k,m)^{*}C_{m}^{\sigma,\beta}(w)\hat{\otimes}M_{n-k}^{\sigma,\beta}$,
where $\mathbf{S}(n,m)$ is the Stirling operator of the second kind
defined in \eqref{eq:Stirling-second-D-property} in Appendix~\ref{sec:Stirling-Operators}
and $M_{n}^{\sigma,\beta}\in(\mathcal{D}_{\mathbb{C}}^{^{\hat{\otimes}n}})'$
are the moment kernels of $\pi_{\sigma}^{\beta}$ given in \eqref{eq:Laplace-transform-inf-dim}.
\item [{(P3)}] \emph{$C_{n}^{\sigma,\beta}(z+w)=\sum_{k+l+m=n}\frac{n!}{k!l!m!}C_{k}^{\sigma,\beta}(z)\hat{\otimes}C_{l}^{\sigma,\beta}(w)\hat{\otimes}M_{m}^{\sigma,\beta,\alpha}$,
}where $M_{m}^{\sigma,\beta,\alpha}\in(\mathcal{D}_{\mathbb{C}}^{^{\hat{\otimes}n}})'$
is determined by 
\begin{equation}
l_{\pi_{\sigma}^{\beta}}(\alpha(\varphi))=\sum_{m=0}^{\infty}\frac{1}{m!}\langle M_{m}^{\sigma,\beta,\alpha},\varphi^{\otimes m}\rangle,\quad\varphi\in\mathcal{D}_{\mathbb{C}}.\label{eq:moments-alpha}
\end{equation}
\item [{(P4)}] \emph{$C_{n}^{\sigma,\beta}(z+w)=\sum_{k=0}^{n}\binom{n}{k}C_{k}^{\sigma,\beta}(z)\hat{\otimes}(w)_{n-k}$,
}where $(w)_{n}$ is the falling factorial on $\mathcal{D}'_{\mathbb{C}}$
determined by \eqref{eq:generating-function-f-factorial-inf-dim}.
\item [{(P5)}] $C_{n}^{\sigma,\beta}(w)=\sum_{k=0}^{n}\binom{n}{k}\sum_{m=0}^{n-k}C_{k}^{\sigma,\beta}(0)\hat{\otimes}\big(\mathbf{s}(n-k,m)^{*}w^{\otimes m}\big)$.
\item [{(P6)}] \emph{$\mathbb{E}_{\pi_{\sigma}^{\beta}}(\langle C_{n}^{\sigma,\beta}(\cdot),\varphi^{(n)}\rangle)=\delta_{n,0}$,
}where $\varphi^{(n)}\in\mathcal{D}_{\mathbb{C}}^{\hat{\otimes}n}$,
$\delta_{n,k}$ is the Kronecker delta function and $\mathbb{E}_{\pi_{\sigma}^{\beta}}(\cdot)$
is the expectation with respect to the measure \emph{$\pi_{\sigma}^{\beta}$.}
\item [{(P7)}] For all $p'>p$ such that the embedding $\mathcal{H}_{p'}\hookrightarrow\mathcal{H}_{p}$
is a Hilbert-Schmidt operator and for all $\varepsilon>0$ there exist
$C_{\varepsilon}>0$ such that 
\[
|C_{n}^{\sigma,\beta}(w)|_{-p'}\leq C_{\varepsilon}n!\varepsilon^{-n}\exp(\varepsilon|w|_{-p}),\quad w\in\mathcal{H}_{-p',\mathbb{C}},\,n\in\mathbb{N}_{0}.
\]
\end{description}
\end{prop}

\begin{proof}
(P1) In view of Equation~\eqref{eq:Wick-Appell-inf-dim}, we have
\begin{equation}
\mathrm{e}_{\pi_{\sigma}^{\beta}}(\alpha(\varphi);w)=\frac{\exp(\langle w,\alpha(\varphi)\rangle)}{l_{\pi_{\sigma}^{\beta}}(\alpha(\varphi))}=\sum_{m=0}^{\infty}\frac{1}{m!}\langle P_{m}^{\sigma,\beta}(w),\alpha(\varphi)^{\otimes m}\rangle.\label{eq:normalized-exp-P-alpha}
\end{equation}
Using Equation \eqref{eq:decomposition-alpha^k}, we obtain 
\begin{align*}
\mathrm{e}_{\pi_{\sigma}^{\beta}}(\alpha(\varphi);w) & =\sum_{m=0}^{\infty}\Bigg\langle P_{m}^{\sigma,\beta}(w),\sum_{n=m}^{\infty}\frac{1}{n!}\mathbf{s}(n,m)\varphi^{\otimes n}\Bigg\rangle\\
 & =\sum_{m=0}^{\infty}\sum_{n=m}^{\infty}\frac{1}{n!}\Bigg\langle\mathbf{s}(n,m)^{*}P_{m}^{\sigma,\beta}(w),\varphi^{\otimes n}\Bigg\rangle\\
 & =\sum_{n=0}^{\infty}\frac{1}{n!}\Bigg\langle\sum_{m=0}^{n}\mathbf{s}(n,m)^{*}P_{m}^{\sigma,\beta}(w),\varphi^{\otimes n}\Bigg\rangle.
\end{align*}
On the other hand, using the equality \eqref{eq:normalized-exp-fPm-inf-dim}
and comparing both series for $\mathrm{e}_{\pi_{\sigma}^{\beta}}(\alpha(\varphi),w)$
gives 
\[
C_{n}^{\sigma,\beta}(w)=\sum_{m=1}^{n}\mathbf{s}(n,m)^{*}P_{m}^{\sigma,\beta}(w).
\]

\noindent (P2) Similar as in the proof of (P1), we use Equation \eqref{eq:normalized-exp-fPm-inf-dim}
and the fact that $g_{\alpha}$ is the inverse of $\alpha$ to obtain
\begin{equation}
\mathrm{e}_{\pi_{\sigma}^{\beta}}(\varphi;w)=\sum_{m=0}^{\infty}\frac{1}{m!}\langle C_{m}^{\sigma,\beta}(w),g_{\alpha}(\varphi)^{\otimes m}\rangle.\label{eq:normalized-exp-C-g}
\end{equation}
Using Equation \eqref{eq:decomposition-g-alpha^k} we replace $g_{\alpha}(\varphi)^{\otimes m}$
in the above Equation \eqref{eq:normalized-exp-C-g} and making some
standard manipulations yields
\[
\mathrm{e}_{\pi_{\sigma}^{\beta}}(\varphi;z)=\sum_{m=0}^{\infty}\Bigg\langle C_{m}^{\sigma,\beta}(w),\sum_{n=m}^{\infty}\frac{1}{n!}\mathbf{S}(n,m)\varphi^{\otimes n}\Bigg\rangle=\sum_{n=0}^{\infty}\frac{1}{n!}\Bigg\langle\sum_{m=0}^{n}\mathbf{S}(n,m)^{*}C_{m}^{\sigma,\beta}(w),\varphi^{\otimes n}\Bigg\rangle.
\]
On the other hand, comparing the above series for $\mathrm{e}_{\pi_{\sigma}^{\beta}}(\varphi,w)$
and the Equation~\eqref{eq:Wick-Appell-inf-dim}, we obtain
\begin{equation}
P_{n}^{\sigma,\beta}(w)=\sum_{m=0}^{n}\mathbf{S}(n,m)^{*}C_{m}^{\sigma,\beta}(w).\label{eq:P-C-polynomials}
\end{equation}
By Equation \eqref{eq:normalized-expo-inf-dim}, we have the equality
\begin{equation}
\mathrm{e}^{\langle w,\varphi\rangle}=\mathrm{e}_{\pi_{\sigma}^{\beta}}(\varphi;w)l_{\pi_{\sigma}^{\beta}}(\varphi).\label{eq:normalized-expo-inf-dim-Laplace}
\end{equation}
Now using the equations \eqref{eq:Wick-Appell-inf-dim} and \eqref{eq:Laplace-transform-inf-dim},
we obtain the equation 
\[
\sum_{n=0}^{\infty}\frac{1}{n!}\langle w^{\otimes n},\varphi^{\otimes n}\rangle=\sum_{n=0}^{\infty}\frac{1}{n!}\Bigg\langle\sum_{k=0}^{n}\binom{n}{k}P_{n}^{\sigma,\beta}(w)\hat{\otimes}M_{n-k}^{\sigma,\beta},\varphi^{\otimes n}\Bigg\rangle
\]
which implies that 
\begin{equation}
w^{\otimes n}=\sum_{k=0}^{n}\binom{n}{k}P_{n}^{\sigma,\beta}(w)\hat{\otimes}M_{n-k}^{\sigma,\beta}.\label{eq:z-P-M-polynomials}
\end{equation}
The claim follows by applying Equation~\eqref{eq:P-C-polynomials}
to Equation~\eqref{eq:z-P-M-polynomials}.

\noindent (P3) By definition of the modified normalized exponential,
we have 
\[
\mathrm{e}_{\pi_{\sigma}^{\beta}}(\alpha(\varphi);z+w)=\mathrm{e}_{\pi_{\sigma}^{\beta}}(\alpha(\varphi);z)\mathrm{e}_{\pi_{\sigma}^{\beta}}(\alpha(\varphi);w)l_{\pi_{\sigma}^{\beta}}(\alpha(\varphi)).
\]
For $l_{\pi_{\sigma}^{\beta}}(\alpha(\varphi))$, we use the decomposition
\eqref{eq:moments-alpha} such that the above equation yields 
\begin{align*}
\sum_{n=0}^{\infty}\frac{1}{n!}\langle C_{n}^{\sigma,\beta}(z+w),\varphi^{\otimes n}\rangle & =\sum_{k=0}^{\infty}\frac{1}{k!}\langle C_{k}^{\sigma,\beta}(z),\varphi^{\otimes k}\rangle\sum_{l=0}^{\infty}\frac{1}{l!}\langle C_{l}^{\sigma,\beta}(w),\varphi^{\otimes l}\rangle\sum_{m=0}^{\infty}\frac{1}{m!}\langle M_{m}^{\sigma,\beta,\alpha},\varphi^{\otimes m}\rangle\\
 & =\sum_{n=0}^{\infty}\frac{1}{n!}\Bigg\langle\sum_{k+l+m=n}\frac{n!}{k!l!m!}C_{k}^{\sigma,\beta}(z)\hat{\otimes}C_{l}^{\sigma,\beta}(w)\hat{\otimes}M_{m}^{\sigma,\beta,\alpha},\varphi^{\otimes n}\Bigg\rangle.
\end{align*}
Thus, the result follows by comparing the coefficients in both sides
of the equation.

\noindent (P4) Again, by definition of the modified normalized exponential,
we have 
\[
\mathrm{e}_{\pi_{\sigma}^{\beta}}(\alpha(\varphi);z+w)=\mathrm{e}_{\pi_{\sigma}^{\beta}}(\alpha(\varphi);z)\exp(\langle w,\alpha(\varphi)\rangle).
\]
By Equations \eqref{eq:generating-function-f-factorial-inf-dim} and
\eqref{eq:normalized-exp-fPm-inf-dim}, we have
\begin{align*}
\sum_{n=0}^{\infty}\frac{1}{n!}\langle C_{n}^{\sigma,\beta}(z+w),\varphi^{\otimes n}\rangle & =\sum_{k=0}^{\infty}\frac{1}{k!}\langle C_{k}^{\sigma,\beta}(z),\varphi^{\otimes k}\rangle\sum_{m=0}^{\infty}\frac{1}{m!}\langle(w)_{m},\varphi^{\otimes m}\rangle\\
 & =\sum_{n=0}^{\infty}\frac{1}{n!}\Bigg\langle\sum_{k=0}^{n}\binom{n}{k}C_{k}^{\sigma,\beta}(z)\hat{\otimes}(w)_{n-k},\varphi^{\otimes n}\Bigg\rangle.
\end{align*}
Thus the assertion follows immediately by comparing the coefficients
in both sides of the equation.

\noindent (P5) The result follows from (P4) at $z=0$ and \eqref{eq:falling-factorial-Stirling}.

\noindent (P6) Note that for $\varphi\in\mathcal{D}_{\mathbb{C}}$,
we have 
\[
\sum_{n=0}^{\infty}\frac{1}{n!}\mathbb{E_{\pi_{\sigma}^{\beta}}}(\langle C_{n}^{\sigma,\beta}(w),\varphi^{\otimes n}\rangle)=\mathbb{E}_{\pi_{\sigma}^{\beta}}(\mathrm{e}_{\pi_{\sigma}^{\beta}}(\alpha(\varphi);w))=\frac{\mathbb{E}_{\pi_{\sigma}^{\beta}}(\exp(\langle\cdot,\alpha(\varphi)\rangle))}{l_{\pi_{\sigma}^{\beta}}(\alpha(\varphi))}=1.
\]
By polarization identity and comparison of coefficients, we obtain
the result.

\noindent (P7) Let $\varepsilon>0$ be given. Then let $C_{\varepsilon},\sigma_{\varepsilon}>0$
be chosen in such a way that $|\alpha(\varphi)|_{p}\leq\varepsilon$
and $C_{\varepsilon}\geq1/|l_{\pi_{\beta}^{\sigma}}(\alpha(\varphi))|$
for $|\varphi|_{p}=\sigma_{\varepsilon}.$ By definition of $C_{n}^{\sigma,\beta}(w)$
and the Cauchy formula, we have
\begin{align*}
|\langle C_{n}^{\sigma,\beta}(w),\varphi^{\otimes n}\rangle| & =|\widehat{\mathrm{d}^{n}\mathrm{e}_{\pi_{\sigma}^{\beta}}(0;w)}(\varphi)|\\
 & \leq n!\frac{1}{\sigma_{\varepsilon}^{n}}\left(\sup_{|\varphi|_{p}=\sigma_{\varepsilon}}\frac{\exp\left(|\alpha(\varphi)|_{p}|w|_{-p}\right)}{|l_{\pi_{\sigma}^{\beta}}(\alpha(\varphi))|}\right)|\varphi|_{p}^{n}\\
 & \leq n!\frac{1}{\sigma_{\varepsilon}^{n}}\left(\sup_{|\varphi|_{p}=\sigma_{\varepsilon}}\frac{1}{|l_{\pi_{\sigma}^{\beta}}(\alpha(\varphi))|}\right)\exp\left(\varepsilon|w|_{-p}\right)|\varphi|_{p}^{n}\\
 & \leq C_{\varepsilon}n!\sigma_{\varepsilon}^{-n}\exp\left(\varepsilon|w|_{-p}\right)|\varphi|_{p}^{n}.
\end{align*}
Let $p'>p$ be such that $i_{p',p}$ is a Hilbert-Schmidt operator.
Then by the kernel theorem, we have

\[
|C_{n}^{\sigma,\beta}(w)|_{-p'}\leq n!C_{\varepsilon}\exp\left(\varepsilon|w|_{-p}\right)\left(\frac{1}{\sigma_{\varepsilon}}\|i_{p',p}\|_{HS}\right)^{n},\quad w\in\mathcal{H}_{-p,\mathbb{C}}.
\]
For sufficiently small $\varepsilon$, we fix $\sigma_{\varepsilon}=\varepsilon\|i_{p',p}\|_{HS}$
so that
\[
|C_{n}^{\sigma,\beta}(w)|_{-p'}\leq n!C_{\varepsilon}\varepsilon^{-n}\exp\left(\varepsilon|w|_{-p}\right).
\]
This concludes the proof.
\end{proof}

\subsection{Generalized Dual Appell System}

\label{sec:Generalized-Dual-Appell-System}In what follows, we use
again the approach in \cite{KSWY95} of non-Gausian analysis to introduce
the generalized dual Appell system associated with the fPm $\pi_{\sigma}^{\beta}$.
\begin{defn}
The \emph{space of smooth polynomials} $\mathcal{P}(\mathcal{D}')$
on $\mathcal{D}'$ is the space consisting of finite linear combinations
of monomial functions, that is, 
\[
\mathcal{P}(\mathcal{D}'):=\left\{ \varphi(w)=\sum_{n=0}^{N(\varphi)}\langle w^{\otimes n},\varphi^{(n)}\rangle\,\bigg|\,\varphi^{(n)}\in\mathcal{D}_{\mathbb{C}}^{\hat{\otimes}n},\;w\in\mathcal{D}',\;N(\varphi)\in\mathbb{N}_{0}\right\} .
\]
\end{defn}

The space $\mathcal{P}(\mathcal{D}')$ shall be equipped with the
natural topology, such that the mapping 
\[
I:\mathcal{P}(\mathcal{D}')\longrightarrow\bigoplus_{n=0}^{\infty}\mathcal{D}_{\mathbb{C}}^{\hat{\otimes}n}
\]
defined for any $\varphi(\cdot)=\sum_{n=0}^{\infty}\langle\cdot^{\otimes n},\varphi^{(n)}\rangle\in\mathcal{P}(\mathcal{D}')$
by 
\[
I\varphi=\vec{\varphi}=(\varphi^{(0)},\varphi^{(1)},\dots,\varphi^{(n)},\dots)
\]
becomes a topological isomorphism from $\mathcal{P}(\mathcal{D}')$
to the topological direct sum of symmetric tensor powers $\mathcal{D}_{\mathbb{C}}^{\hat{\otimes}n}$
(see \cite{BK88,S71}). Note that only a finite number of $\varphi^{(n)}$
is non-zero. With respect to this topology, a sequence $(\varphi_{m})_{m\in\mathbb{N}}$
of smooth continuous polynomials, that is, $\varphi_{m}(w)=\sum_{n=0}^{N(\varphi_{m})}\langle w^{\otimes n},\varphi_{m}^{(n)}\rangle$
converges to $\varphi(w)=\sum_{n=0}^{N(\varphi)}\langle w^{\otimes n},\varphi^{(n)}\rangle\in\mathcal{P}(\mathcal{D}')$
if, and only if, the sequence $(N(\varphi_{m}))_{m\in\mathbb{N}}$
is bounded and $(\varphi_{m}^{(n)})_{m\in\mathbb{N}}$ converges to
$\varphi^{(n)}$ in $\mathcal{D}_{\mathbb{C}}^{\hat{\otimes}n}$ for
all $n\in\mathbb{N}_{0}$.

Using Proposition~\ref{prop:generalized-appell-polynomials-inf-dim}-(P2),
the space of smooth polynomials $\mathcal{P}(\mathcal{D}')$ can also
be expressed in terms of the generalized Appell polynomials associated
with the measure $\pi_{\sigma}^{\beta}$ given by 
\[
\mathcal{P}(\mathcal{D}'):=\left\{ \varphi(w)=\sum_{n=0}^{N(\varphi)}\langle C_{n}^{\sigma,\beta}(w),\varphi^{(n)}\rangle\,\bigg|\,\varphi^{(n)}\in\mathcal{D}_{\mathbb{C}}^{\hat{\otimes}n},w\in\mathcal{D}',N(\varphi)\in\mathbb{N}_{0}\right\} .
\]

We denote by $\mathcal{P}'_{\pi_{\sigma}^{\beta}}(\mathcal{D}')$
the dual space of $\mathcal{P}(\mathcal{D}')$ with respect to $L^{2}(\pi_{\sigma}^{\beta}):=L^{2}(\mathcal{D}',\mathcal{C}_{\sigma}(\mathcal{D}'),\pi_{\sigma}^{\beta};\mathbb{C})$
and obtain the triple 
\begin{equation}
\mathcal{P}(\mathcal{D}')\subset L^{2}(\pi_{\sigma}^{\beta})\subset\mathcal{P}'_{\pi_{\sigma}^{\beta}}(\mathcal{D}').\label{eq:triple-P(D)-L2}
\end{equation}
The (bilinear) dual pairing $\langle\!\langle\cdot,\cdot\rangle\!\rangle_{\pi_{\sigma}^{\beta}}$
between $\mathcal{P}(\mathcal{D}')$ and $\mathcal{P}'_{\pi_{\sigma}^{\beta}}(\mathcal{D}')$
is then related to the (sesquilinear) inner product on $L^{2}(\pi_{\sigma}^{\beta})$
by 
\[
\langle\!\langle F,\varphi\rangle\!\rangle_{\pi_{\sigma}^{\beta}}=(\!(F,\bar{\varphi})\!)_{L^{2}(\pi_{\sigma}^{\beta})},\quad F\in L^{2}(\pi_{\sigma}^{\beta}),\;\varphi\in\mathcal{P}(\mathcal{D}'),
\]
where $\bar{\varphi}$ denotes the complex conjugate function of $\varphi$.
Further we introduce the constant function $\boldsymbol{1}\in L^{2}(\pi_{\sigma}^{\beta})\subset\mathcal{P}'_{\pi_{\sigma}^{\beta}}(\mathcal{D}')$
such that $\boldsymbol{1}(w)=1$ for all $w\in\mathcal{D}'$, so for
any polynomial $\varphi\in\mathcal{P}(\mathcal{D}')$,
\[
\mathbb{E_{\pi_{\sigma}^{\beta}}}(\varphi):=\int_{\mathcal{D}'}\varphi(w)\,\mathrm{d}\pi_{\sigma}^{\beta}(w)=\langle\!\langle\boldsymbol{1},\varphi\rangle\!\rangle_{\pi_{\sigma}^{\beta}}.
\]

Now, we will describe the distributions in $\mathcal{P}'_{\pi_{\sigma}^{\beta}}(\mathcal{D}')$
in a similar way as the smooth polynomials $\mathcal{P}(\mathcal{D}')$,
that is, for any $\Phi\in\mathcal{P}'_{\pi_{\sigma}^{\beta}}(\mathcal{D}')$,
we find elements $\Phi^{(n)}\in(\mathcal{D}_{\mathbb{C}}^{\hat{\otimes}n})'$
and operators $Q_{n}^{\sigma,\beta,\alpha}$ on $(\mathcal{D}_{\mathbb{C}}^{\hat{\otimes}n})'$,
such that 
\[
\Phi=\sum_{n=0}^{\infty}Q_{n}^{\sigma,\beta,\alpha}(\Phi^{(n)})\in\mathcal{P}'_{\pi_{\sigma}^{\beta}}(\mathcal{D}').
\]
To this end, we define first a differential operator $D(\Phi^{(n)})$
depending on $\Phi^{(n)}\in(\mathcal{D}_{\mathbb{C}}^{\hat{\otimes}n})'$
such that when applied to the monomials $\langle w^{\otimes m},\varphi^{(m)}\rangle,$
$\varphi^{(m)}\in\mathcal{D}_{\mathbb{C}}^{\hat{\otimes}m}$, $m\in\mathbb{N}_{0}$,
gives 
\[
D(\Phi^{(n)})\langle w^{\otimes m},\varphi^{(m)}\rangle:=\begin{cases}
\frac{m!}{(m-n)!}\langle w^{\otimes(m-n)}\hat{\otimes}\Phi^{(n)},\varphi^{(m)}\rangle, & \mathrm{for}\:m\geq n\\
0, & \mathrm{otherwise}
\end{cases}
\]
and extend by linearity from the monomials to elements in $\mathcal{P}(\mathcal{D}')$.
If we consider the space of Schwartz test function $\mathcal{S}(\mathbb{R})$
instead of using the space $\mathcal{D}$ with the triple
\[
\mathcal{S}(\mathbb{R})\subset L^{2}(\mathbb{R},\mathrm{d}x)\subset\mathcal{S}'(\mathbb{R}),
\]
then for $n=1$ and $\Phi^{(1)}=\delta_{t}\in\mathcal{S}_{\mathbb{C}}'(\mathbb{R})$,
the differential operator $D(\delta_{t})$ coincides with the Hida
derivative, see \cite{HKPS93}. Note that $D(\Phi^{(n)})$ is a continuous
linear operator from $\mathcal{P}(\mathcal{D}')$ to $\mathcal{P}(\mathcal{D}')$
(see \cite[Lemma 4.13]{KSWY95}) and this enables us to define the
dual operator 
\[
D(\Phi^{(n)})^{*}:\mathcal{P}'_{\pi_{\sigma}^{\beta}}(\mathcal{D}')\longrightarrow\mathcal{P}'_{\pi_{\sigma}^{\beta}}(\mathcal{D}').
\]
Below we need the evaluation of the operator $D(\Phi^{(n)})$ on the
monomials $\langle P_{m}^{\sigma,\beta}(w),\varphi^{(m)}\rangle$,
$m\in\mathbb{N}_{0}$ in \eqref{eq:P-system}. We state this result
in the next proposition and the proof can be found in \cite[Lemma 4.14]{KSWY95}.
\begin{prop}
\label{prop:operator-D-on-P}For $\Phi^{(n)}\in(\mathcal{D}_{\mathbb{C}}^{\hat{\otimes}n})'$
and $\varphi^{(m)}\in\mathcal{D}_{\mathbb{C}}^{\hat{\otimes}m}$ we
have

\[
D(\Phi^{(n)})\langle P_{m}^{\sigma,\beta}(w),\varphi^{(m)}\rangle=\begin{cases}
{\displaystyle \frac{m!}{(m-n)!}\langle P_{m-n}^{\sigma,\beta}(w)\hat{\otimes}\Phi^{(n)},\varphi^{(m)}\rangle,} & \mathit{for}\;m\geq n\\
0, & \mathit{for}\;m<n.
\end{cases}
\]
\end{prop}

Now, we set $Q_{n}^{\sigma,\beta}(\Phi^{(n)}):=D(\Phi^{(n)})^{*}\mathbf{1}$
for $\Phi^{(n)}\in(\mathcal{D}_{\mathbb{C}}^{\hat{\otimes}n})'$ and
denote the so-called $\mathbb{Q}^{\sigma,\beta}$-system in $\mathcal{P}'_{\pi_{\sigma}^{\beta}}(\mathcal{D}')$
by 
\[
\mathbb{Q}^{\sigma,\beta}:=\left\{ Q_{n}^{\sigma,\beta}(\Phi^{(n)})\mid\Phi^{(n)}\in(\mathcal{D}_{\mathbb{C}}^{\hat{\otimes}n})',\,n\in\mathbb{N}_{0}\right\} .
\]

The pair $\mathbb{A}^{\sigma,\beta}=(\mathbb{P}^{\sigma,\beta},\mathbb{Q}^{\sigma,\beta})$
is called the \emph{Appell system generated by the measure} $\pi_{\sigma}^{\beta}$.
This system satisfies the biorthogonal property, see \cite{KSWY95},
given in the following theorem.
\begin{thm}
For $\Phi^{(m)}\in(\mathcal{D}_{\mathbb{C}}^{\hat{\otimes}m})'$ and
$\varphi^{(n)}\in\mathcal{D}_{\mathbb{C}}^{\hat{\otimes}n}$ we have
\begin{equation}
\langle\!\langle Q_{n}^{\sigma,\beta}(\Phi^{(m)}),\langle P_{n}^{\sigma,\beta},\varphi^{(n)}\rangle\rangle\!\rangle_{\pi_{\sigma}^{\beta}}=\delta_{m,n}n!\langle\Phi^{(n)},\varphi^{(n)}\rangle,\quad n,m\in\mathbb{N}_{0}.\label{eq:biorthogonal-P}
\end{equation}
\end{thm}

However, our aim is to construct the generalized dual Appell system
$\mathbb{Q}^{\sigma,\beta,\alpha}$ such that $\mathbb{P}^{\sigma,\beta,\alpha}$
and $\mathbb{Q}^{\sigma,\beta,\alpha}$ are biorthogonal. The reason
to do this is because when $\beta=1$ we obtain only one system of
orthogonal polynomials, so-called the Charlier polynomials, see \cite{IK88}.

First, recall the function $g_{\alpha}(\varphi)$, $\varphi\in\mathcal{D}_{\mathbb{C}}$
from Section \ref{sec:Generalized-Appell-Polynomials}. By Equation
\eqref{eq:decomposition-g-alpha^k}, we have 
\[
g_{\alpha}(\varphi)^{\otimes n}=\sum_{k=n}^{\infty}\frac{n!}{k!}\mathbf{S}(n,k)\varphi^{\otimes k}.
\]
Then for any $\Phi^{(n)}\in(\mathcal{D}_{\mathbb{C}}^{\hat{\otimes}n})'$,
we have 
\begin{equation}
\langle\Phi^{(n)},g_{\alpha}(\varphi)^{\otimes n}\rangle=\sum_{k=n}^{\infty}\frac{n!}{k!}\langle\Phi^{(n)},\mathbf{S}(k,n)\varphi^{\otimes k}\rangle=\sum_{k=n}^{\infty}\frac{n!}{k!}\langle\mathbf{S}(k,n)^{*}\Phi^{(n)},\varphi^{\otimes k}\rangle,\label{eq:Phi-g-alpha-n-Stirling}
\end{equation}
where $\mathbf{S}(k,n)^{*}\Phi^{(n)}\in(\mathcal{D}_{\mathbb{C}}^{\hat{\otimes}k})'$.
Now, we define the operator $G(\Phi^{(n)})$ by
\[
G(\Phi^{(n)}):\mathcal{P}(\mathcal{D}')\longrightarrow\mathcal{P}(\mathcal{D}'),\;\varphi\mapsto G(\Phi^{(n)})\varphi:=\sum_{k=n}^{\infty}\frac{n!}{k!}D(\mathbf{S}(k,n)^{*}\Phi^{(n)})\varphi.
\]
Since $D(\mathbf{S}(k,n)^{*}\Phi^{(n)})$ is continuous for any $\Phi^{(n)}\in(\mathcal{D}_{\mathbb{C}}^{\hat{\otimes}n})'$,
it is easy to see that $G(\Phi^{(n)})$ is also continuous and so
its adjoint $G(\Phi^{(n)})^{*}:\mathcal{P}'_{\pi_{\sigma}^{\beta}}(\mathcal{D}')\longrightarrow\mathcal{P}'_{\pi_{\sigma}^{\beta}}(\mathcal{D}')$
exists.
\begin{defn}
For any $\Phi^{(n)}\in(\mathcal{D}_{\mathbb{C}}^{\hat{\otimes}n})'$,
we define the \emph{generalized function} $Q_{n}^{\sigma,\beta,\alpha}(\Phi^{(n)})\in\mathcal{P}'_{\pi_{\sigma}^{\beta}}(\mathcal{D}')$,
$n\in\mathbb{N}_{0}$, by
\begin{equation}
Q_{n}^{\sigma,\beta,\alpha}(\Phi^{(n)}):=G(\Phi^{(n)})^{*}\boldsymbol{1}.\label{eq:generalized-function-inf-dim}
\end{equation}
The family
\[
\mathbb{Q}^{\sigma,\beta,\alpha}:=\left\{ Q_{n}^{\sigma,\beta,\alpha}(\Phi^{(n)})\mid\Phi^{(n)}\in(\mathcal{D}_{\mathbb{C}}^{\hat{\otimes}n})',\,n\in\mathbb{N}_{0}\right\} 
\]
is said to be the \emph{generalized dual Appell system} $\mathbb{Q}^{\sigma,\beta,\alpha}$
associated with $\pi_{\sigma}^{\beta}$ or the $\mathbb{Q}^{\sigma,\beta,\alpha}$-system
and the pair $\mathbb{A}^{\sigma,\beta,\alpha}:=(\mathbb{P}^{\sigma,\beta,\alpha},\mathbb{Q}^{\sigma,\beta,\alpha})$
is called the \emph{generalized Appell system} generated by the measure
$\pi_{\sigma}^{\beta}$.
\end{defn}

The following theorem states the biorthogonal property of the generalized
Appell system $\mathbb{A}^{\sigma,\beta,\alpha}$.
\begin{thm}
\label{thm:biorthogonal-property-inf-dim}For $\Phi^{(n)}\in(\mathcal{D}_{\mathbb{C}}^{\hat{\otimes}n})'$
and $\varphi^{(m)}\in\mathcal{D}_{\mathbb{C}}^{\hat{\otimes}m}$ we
have 
\[
\langle\!\langle Q_{n}^{\sigma,\beta,\alpha}(\Phi^{(n)}),\langle C_{m}^{\sigma,\beta},\varphi^{(m)}\rangle\rangle\!\rangle_{\pi_{\sigma}^{\beta}}=\delta_{n,m}n!\langle\Phi^{(n)},\varphi^{(n)}\rangle,\quad n,m\in\mathbb{N}_{0}.
\]
\end{thm}

\begin{proof}
By Proposition \ref{prop:generalized-appell-polynomials-inf-dim}-(P1),
we have
\[
\langle C_{m}^{\sigma,\beta},\varphi^{(m)}\rangle=\bigg\langle\sum_{i=0}^{m}\mathbf{s}(m,i)^{*}P_{i}^{\sigma,\beta},\varphi^{(m)}\bigg\rangle=\sum_{i=0}^{m}\langle P_{i}^{\sigma,\beta},\mathbf{s}(m,i)\varphi^{(m)}\rangle.
\]
Then it follows from Proposition \ref{prop:operator-D-on-P} (noted
below with $\star$) and Proposition 11-(P4) in \cite{KSWY95} ($\star$$\star$)
that 
\begin{align*}
\langle\!\langle Q_{n}^{\sigma,\beta,\alpha}(\Phi^{(n)}),\langle C_{m}^{\sigma,\beta},\varphi^{(m)}\rangle\rangle\!\rangle_{\pi_{\sigma}^{\beta}} & =\langle\!\langle\mathbf{1},G(\Phi^{(n)})\langle C_{m}^{\sigma,\beta},\varphi^{(m)}\rangle\rangle\!\rangle_{\pi_{\sigma}^{\beta}}\\
 & =\sum_{i=0}^{m}\langle\!\langle\mathbf{1},G(\Phi^{(n)})\langle P_{i}^{\sigma,\beta},\mathbf{s}(m,i)\varphi^{(m)}\rangle\rangle\!\rangle_{\pi_{\sigma}^{\beta}}\\
 & \overset{\star}{=}\sum_{i=k}^{m}\sum_{k=n}^{\infty}\frac{n!}{k!}\frac{i!}{(i-k)!}\langle\!\langle\mathbf{1},\langle P_{i-k}^{\sigma,\beta}\hat{\otimes}\mathbf{S}(k,n)^{*}\Phi^{(n)},\mathbf{s}(m,i)\varphi^{(m)}\rangle\rangle\!\rangle_{\pi_{\sigma}^{\beta}}\\
 & =\sum_{i=k}^{m}\sum_{k=n}^{\infty}\frac{n!i!}{k!(i-k)!}\mathbb{E}_{\pi_{\sigma}^{\beta}}(\langle P_{i-k}^{\sigma,\beta}\hat{\otimes}\mathbf{S}(k,n)^{*}\Phi^{(n)},\mathbf{s}(m,i)\varphi^{(m)}\rangle)\\
 & \overset{\star\star}{=}\sum_{i=k}^{m}\sum_{k=n}^{m}\frac{n!i!}{k!(i-k)!}\delta_{i,k}\langle\mathbf{S}(k,n)^{*}\Phi^{(n)},\mathbf{s}(m,k)\varphi^{(m)}\rangle\\
 & =\sum_{k=n}^{m}n!\langle\Phi^{(n)},\mathbf{S}(k,n)\mathbf{s}(m,k)\varphi^{(m)}\rangle\\
 & =\delta_{n,m}n!\langle\Phi^{(n)},\varphi^{(n)}\rangle,
\end{align*}
where the last equality is obtained using Proposition \ref{prop:Stirling-1st-2nd-kind-inf-dim}
in Appendix \ref{sec:Stirling-Operators}.
\end{proof}
\begin{rem}
In Appendix \ref{sec:alternative-proof-biorthogonal-property}, we
provide an alternative proof for the biorthogonal property of the
generalized Appell system $\mathbb{A}^{\sigma,\beta,\alpha}$ using
the $S_{\text{\ensuremath{\pi_{\sigma}^{\beta}}}}$-transform (to
be introduced in Section \ref{sec:test-and-generalized-function-spaces})
of the generalized function $Q_{n}^{\sigma,\beta,\alpha}(\Phi^{(n)})\in\mathcal{P}'_{\pi_{\sigma}^{\beta}}(\mathcal{D}')$.
It is based on the fact that $\exp(\langle z,\varphi\rangle)$ is
an eigenfunction of the generalized function $G(\Phi^{(n)})$.
\end{rem}

Using Theorem \ref{thm:biorthogonal-property-inf-dim}, the space
$\mathcal{P}'_{\pi_{\sigma}^{\beta}}(\mathcal{D}')$ can now be characterized
in a similar way as the space $\mathcal{P}(\mathcal{D}')$. See \cite{KdSS98}
for the proof of the following theorem.
\begin{thm}
For every $\Phi\in\mathcal{P}'_{\pi_{\sigma}^{\beta}}(\mathcal{D}')$,
there exists a unique sequence $(\Phi^{(n)})_{n\in\mathbb{N}_{0}}$,
$\Phi^{(n)}\in(\mathcal{D}_{\mathbb{C}}^{\hat{\otimes}n})'$ such
that 
\[
\Phi=\sum_{n=0}^{\infty}Q_{n}^{\sigma,\beta,\alpha}(\Phi^{(n)})
\]
and vice versa, every such series generates a generalized function
in $\mathcal{P}'_{\pi_{\sigma}^{\beta}}(\mathcal{D}')$.
\end{thm}

\section{Test and Generalized Function Spaces}

\label{sec:test-and-generalized-function-spaces}In this section,
we construct the test function space and the generalized function
space associated to the fPm $\pi_{\sigma}^{\beta}$ and study some
properties. Here, we consider a nuclear triple 
\[
\underset{p\in\mathbb{N}}{\mathrm{pr\,lim}}\mathcal{H}_{p}=\mathcal{N}\subset L^{2}(\sigma)\subset\mathcal{N}'=\underset{p\in\mathbb{N}}{\mathrm{ind\,lim}}\mathcal{H}_{-p},
\]
as described in Section~\ref{sec:Nuclear-spaces} such that 
\[
\mathcal{D}\subset\mathcal{N}\subset L^{2}(\sigma)\subset\mathcal{N}'\subset\mathcal{D}'.
\]
Let $\varphi=\sum_{n=0}^{N}\langle C_{n}^{\sigma,\beta}(w),\varphi^{(n)}\rangle\in\mathcal{P}(\mathcal{D}')$
be given. Then we use the fact that $\mathcal{D}\subset\mathcal{N}$
so that $\varphi^{(n)}\in\mathcal{N}_{\mathbb{C}}^{\hat{\otimes}n}$.
Note that
\[
\mathcal{N}_{\mathbb{C}}^{\hat{\otimes}n}=\underset{p\in\mathbb{N}}{\mathrm{pr\,lim}}\,\mathcal{H}_{p,\mathbb{C}}^{\hat{\otimes}n}
\]
and so $\varphi^{(n)}\in\mathcal{H}_{p,\mathbb{C}}^{\hat{\otimes}n}$
for all $p\in\mathbb{N}$. For each $p,q\in\mathbb{N}$ and $\kappa\in[0,1]$,
we introduce a norm $\|\cdot\|_{p,q,\kappa,\pi_{\sigma}^{\beta}}$
on $\mathcal{P}(\mathcal{D}')$ by 
\[
\|\varphi\|_{p,q,\kappa,\pi_{\sigma}^{\beta}}^{2}:=\sum_{n=0}^{\infty}(n!)^{1+\kappa}2^{nq}|\varphi^{(n)}|_{p}^{2}.
\]
Let $(\mathcal{H}_{p})_{q,\pi_{\sigma}^{\beta}}^{\kappa}$ be the
Hilbert space obtained by completing the space $\mathcal{P}(\mathcal{D}')$
with respect to the norm $\|\cdot\|_{p,q,\kappa,\pi_{\sigma}^{\beta}}^ {}$.
The Hilbert space $(\mathcal{H}_{p})_{q,\pi_{\sigma}^{\beta}}^{\kappa}$
has inner product given by 
\[
(\!(\varphi,\psi)\!)_{q,\pi_{\sigma}^{\beta}}^{\kappa}:=\sum_{n=0}^{\infty}(n!)^{1+\kappa}2^{nq}(\varphi^{(n)},\overline{\psi^{(n)}})_{p},
\]
and admits the representation
\[
(\mathcal{H}_{p})_{q,\pi_{\sigma}^{\beta}}^{\kappa}:=\left\{ \varphi=\sum_{n=0}^{\infty}\langle C_{m}^{\sigma,\beta},\varphi^{(n)}\rangle\in L^{2}(\pi_{\sigma}^{\beta})\;\bigg|\;\|\varphi\|_{p,q,\kappa,\pi_{\sigma}^{\beta}}^{2}=\sum_{n=0}^{\infty}(n!)^{1+\kappa}2^{nq}|\varphi^{(n)}|_{p}^{2}<\infty\right\} .
\]
Then the test function space $(\mathcal{N})_{\pi_{\sigma}^{\beta}}^{\kappa}$
is defined by
\[
(\mathcal{N})_{\pi_{\sigma}^{\beta}}^{\kappa}:=\underset{p,q\in\mathbb{N}}{\mathrm{pr\,lim}}(\mathcal{H}_{p})_{q,\pi_{\sigma}^{\beta}}^{\kappa}.
\]

The test function space $(\mathcal{N})_{\pi_{\sigma}^{\beta}}^{\kappa}$
is a nuclear space which is continuously embedded in $L^{2}(\pi_{\sigma}^{\beta})$.

\begin{example}
\label{exa:normalized-exp-as-test-function}The modified normalized
exponential given in \eqref{eq:normalized-expo-inf-dim-alpha} has
the norm 
\[
\|\mathrm{e}_{\pi_{\sigma}^{\beta}}(\alpha(\varphi);\cdot)\|_{p,q,\kappa,\pi_{\sigma}^{\beta}}^{2}=\sum_{n=0}^{\infty}(n!)^{1+\kappa}2^{nq}\frac{|\varphi|_{p}^{2n}}{(n!)^{2}},\quad\varphi\in\mathcal{N}_{\mathbb{C}}.
\]
\begin{enumerate}
\item If $\kappa=0$, we have 
\[
\|\mathrm{e}_{\pi_{\sigma}^{\beta}}(\alpha(\varphi);\cdot)\|_{p,q,0,\pi_{\sigma}^{\beta}}^{2}=\exp(2^{q}|\varphi|_{p}^{2})<\infty,\qquad\forall\varphi\in\mathcal{N}_{\mathbb{C}}.
\]
\item For $\kappa\in(0,1)$, we use the H{\"o}lder inequality with the
pair $(\frac{1}{\kappa},\frac{1}{1-\kappa})$ and obtain 
\begin{align*}
\|\mathrm{e}_{\pi_{\lambda,\beta}}(\alpha(\varphi);\cdot)\|_{p,q,\kappa,\pi_{\sigma}^{\beta}}^{2} & \leq\left(\sum_{n=0}^{\infty}\left(\frac{1}{2^{n\kappa}}\right)^{\frac{1}{\kappa}}\right)^{\kappa}\left(\sum_{n=0}^{\infty}\left(\frac{\left(2^{\kappa}2^{q}|\varphi|_{p}^{2}\right)^{n}}{(n!)^{1-\kappa}}\right)^{\frac{1}{1-\kappa}}\right)^{1-\kappa}\\
 & =2^{\kappa}\exp\left((1-\kappa)2^{\frac{\kappa+q}{1-\kappa}}|\varphi|_{p}^{\frac{2}{1-\kappa}}\right)<\infty,
\end{align*}
for all $\varphi\in\mathcal{N}_{\mathbb{C}}.$ Thus, $\mathrm{e}_{\pi_{\sigma}^{\beta}}(\alpha(\varphi);\cdot)\in(\mathcal{N})_{\pi_{\sigma}^{\beta}}^{\kappa}$,
$\kappa\in[0,1)$.
\item For $\kappa=1$, we have 
\[
\|\mathrm{e}_{\pi_{\sigma}^{\beta}}(\alpha(\varphi);\cdot)\|_{p,q,\kappa,\pi_{\sigma}^{\beta}}^{2}=\sum_{n=0}^{\infty}2^{nq}|\varphi|_{p}^{2n},\quad\varphi\in\mathcal{N}_{\mathbb{C}}.
\]
Hence, we have $\mathrm{e}_{\pi_{\sigma}^{\beta}}(\alpha(\varphi);\cdot)\notin(\mathcal{N})_{\pi_{\sigma}^{\beta}}^{1}$
if $\varphi\neq0$, but $\mathrm{e}_{\pi_{\sigma}^{\beta}}(\alpha(\varphi);\cdot)\in(\mathcal{H}_{p})_{q,\pi_{\sigma}^{\beta}}^{1}$
if $2^{q}|\varphi|_{p}^{2}<1$. Moreover, the set 
\[
\{\mathrm{e}_{\pi_{\sigma}^{\beta}}(\alpha(\varphi);\cdot)\mid2^{q}|\varphi|_{p}^{2}<1,\varphi\in\mathrm{\mathcal{N}_{\mathbb{C}}}\}
\]
is total in $(\mathcal{H}_{p})_{q,\pi_{\sigma}^{\beta}}^{1}$.
\end{enumerate}
\end{example}

For $\kappa=1$, we collect below the most important properties of
the space $(\mathcal{N})_{\pi_{\sigma}^{\beta}}^{1}$, see \cite{KdSS98}
for the proofs.
\begin{thm}
\begin{description}
\item [{$(i)$}] $(\mathcal{N})_{\pi_{\sigma}^{\beta}}^{1}$ is a nuclear
space.
\item [{$(ii)$}] The topology in $(\mathcal{N})_{\pi_{\sigma}^{\beta}}^{1}$
is uniquely defined by the topology on $\mathcal{N}$, i.e., it does
not depend on the choice of the family of norms $\{|\cdot|_{p}\}$,
$p\in\mathbb{N}$.
\item [{$(iii)$}] There exist $p',q'>0$ such that for all $p\geq p'$,
$q\geq q'$ the topological embedding $(\mathcal{H}_{p})_{q,\pi_{\sigma}^{\beta}}^{1}\subset L^{2}(\pi_{\sigma}^{\beta})$
holds. $(\mathcal{N})_{\pi_{\sigma}^{\beta}}^{1}$ is continuously
and densely embedded in $L^{2}(\pi_{\sigma}^{\beta})$.
\end{description}
\end{thm}

\begin{prop}
\label{prop:test-function-estimate}Any test function $\varphi$ in
$(\mathcal{N})_{\pi_{\sigma}^{\beta}}^{1}$ has a uniquely defined
extension to $\mathcal{N}'_{\mathbb{C}}$ as an element of $\mathcal{E}_{\min}^{1}(\mathcal{N}'_{\mathbb{C}})$.
For all $p>p'$ such that the embedding $\mathcal{H}_{p}\hookrightarrow\mathcal{H}_{p'}$
is of the Hilbert-Schmidt class and for all $\varepsilon>0$, $p\in\mathbb{N}$,
we obtain the following bound 
\[
|\varphi(w)|\leq C\|\varphi\|_{p,q,1,\pi_{\sigma}^{\beta}}\mathrm{e}^{\varepsilon|w|_{-p'}},\qquad\varphi\in(\mathcal{N})_{\pi_{\sigma}^{\beta}}^{1},\:w\in\mathcal{H}_{-p,\mathbb{C}},
\]
where $2^{q}>(\varepsilon\|i_{p',p}\|_{HS})^{-2}$ and
\[
C=C_{\varepsilon}(1-2^{-q}\varepsilon^{-2})^{-1/2}.
\]
\end{prop}

For each $p,q\in\mathbb{N}$ and $\kappa\in[0,1]$, we denote by $(\mathcal{H}_{-p})_{-q,\pi_{\sigma}^{\beta}}^{-\kappa}$
the Hilbert space dual of the space $(\mathcal{H}_{p})_{q,\pi_{\sigma}^{\beta}}^{\kappa}$
with respect to $L^{2}(\pi_{\sigma}^{\beta})$ with the corresponding
(Hilbert) norm $\|\cdot\|_{-p,-q,\kappa,\pi_{\sigma}^{\beta}}$. This
space admits the following representation
\[
(\mathcal{H}_{-p})_{-q,\pi_{\sigma}^{\beta}}^{-\kappa}:=\left\{ \Phi=\sum_{n=0}^{\infty}Q_{n}^{\sigma,\beta,\alpha}(\Phi^{(n)})\in\mathcal{P}'_{\pi_{\sigma}^{\beta}}(\mathcal{D}')\,\middle|\,\|\Phi\|_{-p,-q,\kappa,\pi_{\sigma}^{\beta}}^{2}:=\sum_{n=0}^{\infty}(n!)^{1-\kappa}2^{-nq}|\Phi^{(n)}|_{-p}^{2}<\infty\right\} .
\]
By the general duality theory, the dual space $(\mathcal{N})_{\pi_{\sigma}^{\beta}}^{-\kappa}$
of $(\mathcal{N})_{\pi_{\sigma}^{\beta}}^{\kappa}$ with respect to
$L^{2}(\pi_{\sigma}^{\beta})$ is then given by
\[
(\mathcal{N})_{\pi_{\sigma}^{\beta}}^{-\kappa}:=\bigcup_{p,q\in\mathbb{N}}(\mathcal{H}_{-p})_{-q,\pi_{\sigma}^{\beta}}^{-\kappa}.
\]
Since $\mathcal{P}(\mathcal{D}')\subset(\mathcal{N})_{\pi_{\sigma}^{\beta}}^{\kappa}$,
the space $(\mathcal{N})_{\pi_{\sigma}^{\beta}}^{-\kappa}$ can be
viewed as a subspace of $\mathcal{P}'_{\pi_{\sigma}^{\beta}}(\mathcal{D}')$
and so we extend the triple in \eqref{eq:triple-P(D)-L2} to the chain
of spaces
\[
\mathcal{P}(\mathcal{D}')\subset(\mathcal{N})_{\pi_{\sigma}^{\beta}}^{\kappa}\subset L^{2}(\pi_{\sigma}^{\beta})\subset(\mathcal{N})_{\pi_{\sigma}^{\beta}}^{-\kappa}\subset\mathcal{P}'_{\pi_{\sigma}^{\beta}}(\mathcal{D}').
\]
The action of a distribution 
\[
\Phi=\sum_{n=0}^{\infty}Q_{n}^{\sigma,\beta,\alpha}(\Phi^{(n)})\in(\mathcal{N})_{\pi_{\sigma}^{\beta}}^{-\kappa}
\]
on a test function 
\[
\varphi=\sum_{n=0}^{\infty}\langle C_{n}^{\sigma,\beta}(w),\varphi^{(n)}\rangle\in(\mathcal{N})_{\pi_{\sigma}^{\beta}}^{\kappa}
\]
using the biorthogonal property in Theorem \ref{thm:biorthogonal-property-inf-dim}
is given by 
\[
\langle\!\langle\Phi,\varphi\rangle\!\rangle_{\pi_{\sigma}^{\beta}}=\sum_{n=0}^{\infty}n!\langle\Phi^{(n)},\varphi^{(n)}\rangle.
\]

Now we give two examples of the generalized functions in $(\mathcal{N})_{\pi_{\sigma}^{\beta}}^{-1}$.
For a more generalized case, see \cite{KdSS98}.
\begin{example}[Generalized Radon-Nikodym derivative]
\label{exa:RND} We define a generalized function $\rho_{\pi_{\sigma}^{\beta}}^{\alpha}(w,\cdot)\in(\mathcal{N})_{\pi_{\sigma}^{\beta}}^{-1}$,
$w\in\mathcal{N}'_{\mathbb{C}}$ with the following property 
\[
\langle\!\langle\rho_{\pi_{\sigma}^{\beta}}^{\alpha}(w,\cdot),\varphi\rangle\!\rangle_{\pi_{\sigma}^{\beta}}=\int_{\mathcal{N}'}\varphi(x-w)\,\mathrm{d}\mu(x),\quad\varphi\in(\mathcal{N})_{\pi_{\sigma}^{\beta}}^{1}.
\]
First, we have to establish the continuity of $\rho_{\pi_{\sigma}^{\beta}}^{\alpha}(w,\cdot)$.
Let $w\in\mathcal{H}_{-p,\mathbb{C}}$ be given. Then, if $p\geq p'$
is sufficiently large and $\varepsilon>0$ is small enough, we use
Proposition~\ref{prop:test-function-estimate}, that is, there exists
$q\in\mathbb{N}$ and $C>0$ such that 
\begin{align*}
{\displaystyle \left|\int_{\mathcal{N}'}\varphi(x-w)\,\mathrm{d}\pi_{\sigma}^{\beta}(x)\right|} & \leq{\displaystyle C\|\varphi\|_{p,q,1,\pi_{\sigma}^{\beta}}\int_{\mathcal{N}'}\exp(\varepsilon|x-w|_{-p'})\,\mathrm{d}\pi_{\sigma}^{\beta}(x)}\\
 & {\displaystyle \leq C\|\varphi\|_{p,q,1,\pi_{\sigma}^{\beta}}\exp(\varepsilon|w|_{-p'})\int_{\mathcal{N}'}\exp(\varepsilon|x|_{-p'})\,\mathrm{d}\pi_{\sigma}^{\beta}(x).}
\end{align*}
Since $\varepsilon$ is sufficiently small, the last integral exists
by Lemma 9 from \cite{KSWY95}. This implies that $\rho_{\pi_{\sigma}^{\beta}}^{\alpha}(w,\cdot)\in(\mathcal{N})_{\pi_{\sigma}^{\beta}}^{-1}$.
Let us show that in $(\mathcal{N})_{\pi_{\sigma}^{\beta}}^{-1}$ the
generalized function $\rho_{\pi_{\sigma}^{\beta}}^{\alpha}(w,\cdot)$
admits the canonical expansion 
\begin{equation}
\rho_{\pi_{\sigma}^{\beta}}^{\alpha}(w,\cdot)=\sum_{k=0}^{\infty}\frac{1}{k!}Q_{n}^{\sigma,\beta,\alpha}((-w)_{k}).\label{eq:rho-alpha-expansion}
\end{equation}
Note that the right hand side of \eqref{eq:rho-alpha-expansion} defines
an element in $(\mathcal{N})_{\pi_{\sigma}^{\beta}}^{-1}$. Then it
is sufficient to compare the action of both sides of \eqref{eq:rho-alpha-expansion}
on a total set from $(\mathcal{N})_{\pi_{\sigma}^{\beta}}^{1}$. For
$\varphi^{(n)}\in\mathcal{D}_{\mathbb{C}}^{\hat{\otimes}n}$, we use
the biorthogonal property of $\mathbb{P}^{\sigma,\beta,\alpha}$ and
$\mathbb{Q}^{\sigma,\beta,\alpha}$-systems and obtain 
\begin{align*}
\langle\!\langle\rho_{\pi_{\sigma}^{\beta}}^{\alpha}(w,\cdot),\langle C_{n}^{\sigma,\beta},\varphi^{(n)}\rangle\rangle\!\rangle_{\pi_{\sigma}^{\beta}} & =\bigg\langle\!\!\!\bigg\langle\sum_{k=0}^{\infty}\frac{1}{k!}Q_{n}^{\sigma,\beta,\alpha}((-w)_{k}),\langle C_{n}^{\sigma,\beta},\varphi^{(n)}\rangle\bigg\rangle\!\!\!\bigg\rangle_{\pi_{\sigma}^{\beta}}\\
 & =\langle(-w)_{n},\varphi^{(n)}\rangle.
\end{align*}
On the other hand, by Proposition \ref{prop:generalized-appell-polynomials-inf-dim}-(P4)
and (P6),
\begin{align*}
\langle\!\langle\rho_{\pi_{\sigma}^{\beta}}^{\alpha}(w,\cdot),\langle C_{n}^{\sigma,\beta},\varphi^{(n)}\rangle\rangle\!\rangle_{\pi_{\sigma}^{\beta}} & ={\displaystyle \int}_{\mathcal{D}'}\langle C_{n}^{\sigma,\beta}(x-w),\varphi^{(n)}\rangle\,\mathrm{d}\pi_{\sigma}^{\beta}(x)\\
 & {\displaystyle \negthickspace\overset{(P4)}{=}\sum_{k=0}^{\infty}}\binom{n}{k}{\displaystyle \int}_{\mathcal{D}'}\langle C_{n}^{\sigma,\beta}(x)\hat{\otimes}(-w)_{n-k},\varphi^{(n)}\rangle\,\mathrm{d}\pi_{\sigma}^{\beta}(x)\\
 & {\displaystyle =\sum_{k=0}^{\infty}}\binom{n}{k}\mathbb{E}_{\pi_{\sigma}^{\beta}}(\langle C_{n}^{\sigma,\beta}(x)\hat{\otimes}(-w)_{n-k},\varphi^{(n)}\rangle)\\
 & \negthickspace\negthickspace\overset{(P6)}{=}\langle(-w)_{n},\varphi^{(n)}\rangle.
\end{align*}
Thus, we have shown that $\rho_{\pi_{\sigma}^{\beta}}^{\alpha}(w,\cdot)$
is the generating function of the $\mathbb{Q}^{\sigma,\beta,\alpha}$-system,
i.e.,
\[
\rho_{\pi_{\sigma}^{\beta}}^{\alpha}(-w,\cdot)=\sum_{k=0}^{\infty}\frac{1}{k!}Q_{k}^{\sigma,\beta,\alpha}((w)_{k}).
\]
\end{example}

\begin{example}[Delta function]
 For $w\in\mathcal{N}'_{\mathbb{C}}$, we define a distribution by
the following $\mathbb{Q}^{\sigma,\beta,\alpha}$-decomposition:
\[
\delta_{w}=\sum_{n=0}^{\infty}Q_{n}^{\sigma,\beta,\alpha}(C_{n}^{\sigma,\beta}(w)).
\]
If $p\in\mathbb{N}$ is large enough and $\varepsilon>0$ is sufficiently
small, by Proposition \ref{prop:generalized-appell-polynomials-inf-dim}-(P7),
for any $w\in\mathcal{H}_{-p,\mathbb{C}}$ we have
\[
\|\delta_{w}\|_{-p,-q,\pi_{\sigma}^{\beta},\alpha}^{2}=\sum_{n=0}^{\infty}2^{-nq}|C_{n}^{\sigma,\beta}(w)|_{-p}^{2}\overset{(P7)}{\leq}C_{\varepsilon}^{2}\exp(2\varepsilon|w|_{-p})\sum_{n=0}^{\infty}\varepsilon^{-2n}2^{-nq}
\]
which is finite for sufficiently large $q\in\mathbb{N}$. This implies
that $\delta_{w}\in(\mathcal{N})_{\pi_{\sigma}^{\beta}}^{-1}$. Now,
for 
\[
\varphi=\sum_{n=0}^{\infty}\langle C_{n}^{\sigma,\beta},\varphi^{(n)}\rangle\in(\mathcal{N})_{\pi_{\sigma}^{\beta}}^{1}
\]
 the action of $\delta_{w}$ is given by 
\[
\langle\!\langle\delta_{w},\varphi\rangle\!\rangle_{\pi_{\sigma}^{\beta}}=\sum_{n=0}^{\infty}\langle C_{n}^{\sigma,\beta}(w),\varphi^{(n)}\rangle=\varphi(w)
\]
using the biorthogonal property of $\mathbb{P}^{\sigma,\beta,\alpha}$
and $\mathbb{Q}^{\sigma,\beta,\alpha}$-systems. This means that $\delta_{w}$
(in particular for $w$ real) plays the role of a $\delta$-function
(evaluation map) in the calculus we discuss.
\end{example}

Recall Example \ref{exa:normalized-exp-as-test-function} where the
modified normalized exponential $\mathrm{e}_{\pi_{\sigma}^{\beta}}(\alpha(\varphi);\cdot)$
is a test function in $(\mathcal{N})_{\pi_{\sigma}^{\beta}}^{1}$
only if $2^{q}|\varphi|_{p}^{2}<1$ for $\varphi\in\mathcal{N}_{\mathbb{C}}$.
We define the $S_{\pi_{\sigma}^{\beta}}$-transform of a distribution
$\Phi\in(\mathcal{N})_{\pi_{\sigma}^{\beta}}^{-1}\subset\mathcal{P}'_{\pi_{\sigma}^{\beta}}(\mathcal{D}')$
by 
\[
S_{\pi_{\sigma}^{\beta}}\Phi(\varphi):=\langle\!\langle\Phi,\mathrm{e}_{\pi_{\sigma}^{\beta}}(\varphi;\cdot)\rangle\!\rangle_{\pi_{\sigma}^{\beta}}
\]
if $\varphi$ is chosen in the above way. By the biorthogonal property
of $\mathbb{P}^{\sigma,\beta,\alpha}$ and $\mathbb{Q}^{\sigma,\beta,\alpha}$-systems,
we have 
\[
S_{\pi_{\sigma}^{\beta}}\Phi(\varphi)=\sum_{n=0}^{\infty}\langle\Phi^{(n)},g_{\alpha}(\varphi)^{\otimes n}\rangle.
\]

Now, we introduce the convolution of a function $\varphi\in(N)_{\pi_{\sigma}^{\beta}}^{1}$,
with respect to the measure $\pi_{\sigma}^{\beta}$ given by 
\[
C_{\pi_{\sigma}^{\beta}}\varphi(w)=\langle\!\langle\rho_{\pi_{\sigma}^{\beta}}^{\alpha}(-w,\cdot),\varphi\rangle\!\rangle_{\pi_{\sigma}^{\beta}},
\]
where $\rho_{\pi_{\sigma}^{\beta}}^{\alpha}(-w,\cdot)\in(N)_{\pi_{\sigma}^{\beta}}^{-1}$
is the generalized Radon-Nikodym derivative (see Example~\ref{exa:RND})
for any $w\in\mathcal{N}'_{\mathbb{C}}$. If $\varphi$ has the representation
\[
\varphi=\sum_{n=0}^{\infty}\langle C_{n}^{\sigma,\beta}(w),\varphi^{(n)}\rangle\in(\mathcal{N})_{\pi_{\sigma}^{\beta}}^{1},
\]
then the action of $C_{\pi_{\sigma}^{\beta}}$ on $\varphi$ is given,
for every $w\in\mathcal{N}'_{\mathbb{C}}$, by 
\[
C_{\pi_{\sigma}^{\beta}}\varphi(w)=\sum_{n=0}^{\infty}\langle(-w)_{n},\varphi^{(n)}\rangle\overset{\eqref{eq:Stirling-1st-inner-product}}{=}\sum_{n=0}^{\infty}\sum_{k=0}^{n}(-1)^{k}\langle w^{\otimes k},\mathbf{s}(n,k)\varphi^{(n)}\rangle.
\]
The following is the characterization theorem of the test and generalized
function spaces associated to the fPm which is a standard result for
this approach. For the proof, we refer to \cite{KSWY95,KdSS98}.

\begin{thm}
\label{thm:characterizations}
\begin{description}
\item [{$(i)$}] The convolution $C_{\pi_{\sigma}^{\beta}}$ is a topological
isomorphism from $(\mathcal{N})_{\pi_{\sigma}^{\beta}}^{1}$ on $\mathcal{E}_{\min}^{1}(\mathcal{N}'_{\mathbb{C}})$.
\item [{$(ii)$}] The $S_{\pi_{\sigma}^{\beta}}$-transform is a topological
isomorphism from $(\mathcal{N})_{\pi_{\sigma}^{\beta}}^{-1}$ on $\mathrm{Hol}_{0}(\mathcal{N}_{\mathbb{C}})$.
\item [{$(iii)$}] The $S_{\pi_{\sigma}^{\beta}}$-transform is a topological
isomorphism from $(\mathcal{N})_{\pi_{\sigma}^{\beta}}^{-\kappa}$,
$\kappa\in[0,1)$, on $\mathcal{E}_{\max}^{2/(1-\kappa)}(\mathcal{N}_{\mathbb{C}})$.
\end{description}
\end{thm}

\section{Conclusion and Outlook}

In this paper, we constructed the generalized Appell system $\mathbb{A}^{\sigma,\beta,\alpha}=(\mathbb{P}^{\sigma,\beta,\alpha},\mathbb{Q}^{\sigma,\beta,\alpha})$
associated to the fPm $\pi_{\sigma}^{\beta}$ in infinite dimension.
The Appell polynomials $\mathbb{P}^{\sigma,\beta,\alpha}$ (generated
by the modified Wick exponential) and the dual Appell system $\mathbb{Q}^{\sigma,\beta,\alpha}$
are biorthogonal to each other, see Theorem~\ref{thm:biorthogonal-property-inf-dim}.
It turns out that the kernels $C_{n}^{\sigma,\beta}(\cdot)$ of the
system $\mathbb{P}^{\sigma,\beta,\alpha}$ are given in terms of the
Stirling operators or in terms of the falling factorials on $\mathcal{D}'_{\mathbb{C}}$,
see Proposition~\ref{prop:generalized-appell-polynomials-inf-dim}.
The system $\mathbb{P}^{\sigma,\beta,\alpha}$ is used to define the
spaces of test functions $(\mathcal{N})_{\pi_{\sigma}^{\beta}}^{\kappa}$,
$0\leq\kappa\leq1$, while $\mathbb{Q}^{\sigma,\beta,\alpha}$ is
suitable to describe the generalized functions spaces $(\mathcal{N})_{\pi_{\sigma}^{\beta}}^{-\kappa}$
arising from $\pi_{\sigma}^{\beta}$, see Section~\ref{sec:test-and-generalized-function-spaces}.
The spaces $(\mathcal{N})_{\pi_{\sigma}^{\beta}}^{\kappa}$ and $(\mathcal{N})_{\pi_{\sigma}^{\beta}}^{-\kappa}$
are universal in the sense that their characterization via $S_{\pi_{\sigma}^{\beta}}$-transform
is independent of the measure $\pi_{\sigma}^{\beta}$ (see Theorem~\ref{thm:characterizations})
as is well known from non-Gaussian analysis, \cite{KSWY95,KdSS98}.

In a future work we plan to investigate the stochastic counterpart
associated to the fPm, namely the fractional Poisson process $N_{\lambda}^{\beta}$
in one and infinite dimensions. In particular, their representations
in terms of known processes as well as possible applications.

\appendix

\section{Appendix}

\subsection{Kolmogorov extension theorem on configuration space}

\label{sec:Kolmogorov-extension-theorem-on-config-space} In this
section, we discuss a version of Kolmogorov extension theorem to the
configuration space $(\Gamma,\mathcal{B}(\Gamma))$. The following
definitions and properties of measurable spaces can be found in \cite{C83},
\cite{G88} and \cite{P67}.
\begin{defn}
Let $(X,\mathcal{A})$ and $(X',\mathcal{A}')$ be two measurable
spaces.
\begin{enumerate}
\item The spaces $(X,\mathcal{A})$ and $(X',\mathcal{A}')$ are called
isomorphic if, and only if, there exists a measurable bijective mapping
$f:X\longrightarrow X'$ such that its inverse $f^{-1}$ is also measurable.
\item $(X,\mathcal{A})$ and $(X',\mathcal{A}')$ are called $\sigma$-isomorphic
if, and only if, there exists a bijective mapping $F:\mathcal{A}\longrightarrow\mathcal{A}'$
between the $\sigma$-algebras which preserves the operations in a
$\sigma$-algebra.
\item $(X,\mathcal{A})$ is said to be countable generated if, and only
if, there exists a denumerable class $\mathscr{D}\subset\mathcal{A}$
such that $\mathscr{D}$ generates $\mathcal{A}$.
\item $(X,\mathcal{A})$ is said to be separable if, and only if, it is
countably generated and for each $x\in X$ the set $\left\{ x\right\} \in\mathcal{A}$.
\end{enumerate}
\end{defn}

\begin{defn}
Let $(X,\mathcal{A})$ be a countable generated measurable space.
Then $(X,\mathcal{A})$ is called the standard Borel space if, and
only if, there exists a Polish space $(X',\mathcal{A}')$ (i.e., a
metrizable, complete metric space which fulfills the second axiom
of countability and the $\sigma$-algebra $\mathcal{A}'$ coincides
with the Borel $\sigma$-sigma) such that $(X,\mathcal{A})$ and $(X',\mathcal{B}(X'))$
are $\sigma$-isomorphic.
\end{defn}

\begin{example}
\label{exa:standard-Borel-spaces}
\begin{enumerate}
\item Every locally compact, $\sigma$-compact space is a standard Borel
space.
\item Polish spaces are standard Borel spaces.
\end{enumerate}
\end{example}

\begin{prop}
\label{prop:X-X'-sigma-isomorphic}
\begin{enumerate}
\item If $(X,\mathcal{A})$ is a countable generated measurable space,
then there exists $E\subset\left\{ 0,1\right\} ^{\mathbb{N}}$ such
that $(X,\mathcal{A})$ is $\sigma$-isomorphic to $(E,\mathcal{B}(E))$.
Thus $(X,\mathcal{A})$ is $\sigma$-isomorphic to a separable measurable
space.
\item Let $(X,\mathcal{A})$ and $(X',\mathcal{A}')$ be separable measurable
spaces. Then $(X,\mathcal{A})$ is $\sigma$-isomorphic to $(X',\mathcal{A}')$
if, and only if, they are isomorphic.
\end{enumerate}
\end{prop}

The following theorem states some operations under which separable
standard Borel spaces are closed, see \cite{P67} and \cite{C83}.
\begin{thm}
\label{thm:separable-standard-Borel-space}
\begin{enumerate}
\item Countable product, sum, and union are separable standard Borel spaces.
\item The projective limit is a separable standard Borel space.
\item Any measurable subset of a separable standard Borel space is also
a separable standard Borel space.
\end{enumerate}
\end{thm}

We need also a version of Kolmogorov's extension theorem for separable
standard Borel spaces.
\begin{thm}[{cf. \cite[Chap. V, Theorem 3.2]{P67}}]
\label{thm:Kolmogorov-extension-thm} Let $(X_{n},\mathcal{A}_{n})$,
$n\in\mathbb{N}$, be separable standard Borel spaces. Let $(X,\mathcal{A})$
be the projective limit of the space $(X_{n},\mathcal{A}_{n})$ relative
to the maps $p_{m,n}:X_{n}\longrightarrow X_{m}$, $m\leq n$. If
$\left\{ \mu_{n}\right\} _{n\in\mathbb{N}}$ is a sequence of probability
measures such that n is a measure on $(X_{n},\mathcal{A}_{n})$ and
$\mu_{m}=\mu_{n}\circ p_{m,n}^{-1}$ for $m\leq n$. Then there exists
a unique measure on $(X,\mathcal{A})$ such that $\mu_{n}=\mu\circ p_{n}^{-1}$
for all $n\in\mathbb{N}$ where $p_{n}$ is the projection map from
$X$ on $X_{n}$.
\end{thm}

This theorem can be extended to an index set $I$ which is a directed
set with an order generating sequence, i.e., there exists a sequence
$(\alpha_{n})_{n\in\mathbb{N}}$ in $I$ such that for every $\alpha\in I$
there exists $n\in\mathbb{N}$ with $\alpha<\alpha_{n}$. We apply
this general framework to our configuration space $\Gamma$. Assume
that $(X,\mathfrak{X})$ is a separable standard Borel space. To use
$\mathcal{B}_{c}(X)$ makes this generality have no sense, hence we
have to introduce an abstract concept of local sets. Let $\mathcal{J}_{X}$
be a subset of $\mathfrak{X}$ with the properties:
\begin{description}
\item [{(I1)}] $\Lambda_{1}\cup\Lambda_{2}\in\mathcal{J}_{X}$ for all
$\Lambda_{1},\Lambda_{2}\in\mathcal{J}_{X}.$
\item [{(I2)}] If $\Lambda\in\mathcal{J}_{X}$ and $A\in\mathfrak{X}$
with $A\subset\Lambda$ then $A\in\mathcal{J}_{X}$.
\item [{(I3)}] There exists a sequence $\left\{ \Lambda_{n}\mid n\in\mathbb{N}\right\} $
from $\mathcal{J}_{X}$ with $X=\bigcup_{n\in\mathbb{N}}\Lambda_{n}$
such that if $\Lambda\in\mathcal{J}_{X}$ then $\Lambda\subset\Lambda_{n}$
for some $n\in\mathbb{N}$.
\end{description}
We can then construct the configuration space as in Subsection~\ref{subsec:configuration-space}
taking $X=\mathbb{R}^{d}$ and replacing $\mathcal{B}(\mathbb{R}^{d})$
by $\mathcal{J}_{\mathbb{R}^{d}}$. Our aim is to show that $(\Gamma,\mathcal{B}(\Gamma))$
is a separable standard Borel space and thus by Theorem~\ref{thm:Kolmogorov-extension-thm}
the measure $\pi_{\sigma}^{\beta}$ in Subsection~\ref{sec:fPm-cspace}
exists.

It follows from Theorem~\ref{thm:separable-standard-Borel-space}
that for any $\Lambda\in\mathcal{J}_{\mathbb{R}^{d}}$ and for any
$n\in\mathbb{N}$, the set $\Lambda^{n}$ is a separable standard
Borel space. Thus, by the same argument $\widetilde{\Lambda^{n}}/S_{n}$
is also a separable standard Borel space, see e.g. \cite{S94}. Now
taking into account the isomorphism between $\widetilde{\Lambda^{n}}/S_{n}$
and $\Gamma_{\Lambda}^{(n)}$, $\Gamma_{\Lambda}^{(n)}$ is also a
separable standard Borel space as well as $\Gamma_{\Lambda}$ by Theorem~\ref{thm:separable-standard-Borel-space}-(1).
Therefore, given $(\Gamma,\mathcal{B}(\Gamma))$ as the projective
limit of the projective system $\left\{ (\Gamma_{\Lambda},\mathcal{B}(\Gamma_{\Lambda})),p_{\Lambda_{1},\Lambda_{2}},\mathcal{J}_{\mathbb{R}^{d}}\right\} $
of separable standard Borel spaces, by Theorem~\ref{thm:separable-standard-Borel-space}-(2),
$(\Gamma,\mathcal{B}(\Gamma))$ is a separable standard Borel space.

\subsection{Stirling Operators}

\label{sec:Stirling-Operators}In this appendix we discuss the Stirling
operators which we use in Section~\ref{sec:Appell-System} related
to the Taylor expansion of a holomorphic function typical in Poisson
analysis. For more details and other applications, see \cite{Finkelshtein2019b,Finkelshtein2022}.

For $n\in\mathbb{N}$ and $k\in\mathbb{N}_{0}$, we define the \emph{falling
factorial} by 
\[
(k)_{n}:=n!\binom{k}{n}=k(k-1)\cdots(k-n+1).
\]
The latter expression allows us to define falling factorials as polynomials
of a variable $z\in\mathbb{C}$ replacing $k$ as 
\[
(z)_{n}:=z(z-1)\cdots(z-n+1).
\]
The generating function of the falling factorials is 
\[
\sum_{n=0}^{\infty}\frac{u^{n}}{n!}(z)_{n}=\exp[z\log(1+u)].
\]
The \emph{Stirling numbers of the first kind,} denoted by $s(n,k)$,
are defined as the coefficients of the expansion $(z)_{n}$ in $z$,
in explicit,
\[
(z)_{n}:=\sum_{k=1}^{n}s(n,k)z^{k},
\]
while the \emph{Stirling numbers of the second kind, }denoted by $S(n,k)$,
are defined as the coefficients of the expansion $z^{n}$ in $(z)_{k}$,
that is,
\[
z^{n}=\sum_{k=1}^{n}S(n,k)(z)_{k}.
\]

Let us consider lifting the polynomials $(z)_{n}$ to $\mathcal{D}_{\mathbb{C}}'$.
We call these polynomials the falling factorials on $\mathcal{D}_{\mathbb{C}}'$,
denoted by $(w)_{n}$, for $w\in\mathcal{D}_{\mathbb{C}}'$ (see \cite{Finkelshtein2019b}).
The generating function of the falling factorials on $\mathcal{D}_{\mathbb{C}}'$
is given by 
\begin{equation}
\exp(\langle w,\log(1+\varphi)\rangle)=\sum_{n=0}^{\infty}\frac{1}{n!}\langle(w)_{n},\varphi^{\otimes n}\rangle,\quad\varphi\in\mathcal{D}_{\mathbb{C}},\;w\in\mathcal{D}_{\mathbb{C}}'.\label{eq:generating-function-f-factorial-inf-dim}
\end{equation}
The falling factorial may be written recursively (see Proposition
5.4 in \cite{Finkelshtein2019b}) as follows
\begin{align*}
 & (w)_{0}=1,\\
 & (w)_{1}=w,\\
 & (w)_{n}(x_{1},\dots,x_{n})=w(x_{1})(w(x_{2})-\delta_{x_{1}}(x_{2}))\\
 & \qquad\qquad\times\dots\times(w(x_{n})-\delta_{x_{1}}(x_{n})-\delta_{x_{2}}(x_{n})-\dots-\delta_{x_{n-1}}(x_{n})),
\end{align*}
for $n\geq2$ and $(x_{1},\dots,x_{n})\in(\mathbb{R}^{d})^{n}$.

Now we define the \emph{Stirling operators of the first kind} as
the linear operators $\mathbf{s}(n,k):\mathcal{D}_{\mathbb{C}}^{\hat{\otimes}n}\longrightarrow\mathcal{D}_{\mathbb{C}}^{\hat{\otimes}k}$,
$n\geq k$, satisfying 
\begin{align}
\langle(w)_{n},\varphi^{(n)}\rangle & =\sum_{k=1}^{n}\langle w^{\otimes k},\mathbf{s}(n,k)\varphi^{(n)}\rangle,\qquad\varphi^{(n)}\in\mathcal{D}_{\mathbb{C}}^{\hat{\otimes}n},\;w\in\mathcal{D}_{\mathbb{C}}',\label{eq:Stirling-1st-inner-product}
\end{align}
and the \emph{Stirling operators of the second kind} as the linear
operators $\mathbf{S}(n,k):\mathcal{D}_{\mathbb{C}}^{\hat{\otimes}n}\longrightarrow\mathcal{D}_{\mathbb{C}}^{\hat{\otimes}k}$,
$n\geq k$, satisfying
\begin{align}
\langle w^{\otimes n},\varphi^{(n)}\rangle & =\sum_{k=1}^{n}\langle(w)_{k},\mathbf{S}(n,k)\varphi^{(n)}\rangle,\qquad\varphi^{(n)}\in\mathcal{D}_{\mathbb{C}}^{\hat{\otimes}n},\;w\in\mathcal{D}_{\mathbb{C}}'.\label{eq:Stirling-2nd-inner-product}
\end{align}

\begin{rem}
The Stirling operators $\mathbf{s}(n,k)$ and $\mathbf{S}(n,k)$
introduced in \cite{Finkelshtein2022} are defined on the space of
measurable, bounded, compactly supported, symmetric functions, however;
in this paper, we define these operators on the space $\mathcal{D}_{\mathbb{C}}^{\hat{\otimes}n}$
as a consequence of extending the falling factorials to the space
of generalized functions $\mathcal{D}_{\mathbb{C}}'$ rather than
using the space of Radon measures.
\end{rem}

Let $n,k\in\mathbb{N}$, $k\leq n$ and $i_{1},\dots,i_{k}\in\mathbb{N}$
such that $i_{1}+\dots+i_{k}=n$. We define the operator $\mathbb{D}_{i_{1},\dots,i_{k}}^{(n)}\in\mathcal{L}(\mathcal{D}_{\mathbb{C}}^{\hat{\otimes}n},\mathcal{D}_{\mathbb{C}}^{\hat{\otimes}k})$
(the space of linear operators from $\mathcal{D}_{\mathbb{C}}^{\hat{\otimes}n}$
into $\mathcal{D}_{\mathbb{C}}^{\hat{\otimes}k}$) by 
\begin{equation}
(\mathbb{D}_{i_{1},\dots,i_{k}}^{(n)}\varphi^{(n)})(x_{1},\dots,x_{k}):=\frac{1}{k!}\sum_{\iota\in S_{k}}\varphi^{(n)}(\underbrace{x_{\iota(1)},\dots,x_{\iota(1)}}_{i_{1}},\dots,\underbrace{x_{\iota(k)},\dots,x_{\iota(k)}}_{i_{k}}),\label{eq:derivative-of-order-n}
\end{equation}
for $\varphi^{(n)}\in\mathcal{D}_{\mathbb{C}}^{\hat{\otimes}n}$ and
$(x_{1},\dots,x_{k})\in(\mathbb{R}^{d})^{k}$. In particular, we have
\begin{equation}
\mathbb{D}_{i_{1},\dots,i_{k}}^{(n)}\varphi^{\otimes n}=\varphi^{i_{1}}\hat{\otimes}\dots\hat{\otimes}\varphi^{i_{k}},\quad\varphi\in\mathcal{D}_{\mathbb{C}}.\label{eq:derivative-tensor-product}
\end{equation}
When $k=1$, we denote $\mathbb{D}_{n}^{(n)}:=\mathbb{D}^{(n)}$ such
that for $\varphi^{(n)}\in\mathcal{D}_{\mathbb{C}}^{\hat{\otimes}n}$,
\[
(\mathbb{D}^{(n)}\varphi^{(n)})(x)=\varphi^{(n)}(x,\dots,x),\quad x\in\mathbb{R}^{d}.
\]
The operator $\mathbb{D}_{i_{1},\dots,i_{k}}^{(n)}$ is continuous
(see \cite{Finkelshtein2019b} and \cite{Finkelshtein2022}), that
is, its adjoint $(\mathbb{D}_{i_{1},\dots,i_{k}}^{(n)})^{*}:(\mathcal{D}_{\mathbb{C}}^{\hat{\otimes}k})'\longrightarrow(\mathcal{D}_{\mathbb{C}}^{\hat{\otimes}n})'$
exists and is well-defined. In fact, the operators $\mathbf{s}(n,k)$
and $\mathbf{S}(n,k)$ can be written explicitly in terms of the operator
$\mathbb{D}_{i_{1},\dots,i_{k}}^{(n)}$, (see Proposition 3.7 in \cite{Finkelshtein2022})
that is, for any $n,k\in\mathbb{N}$, $k\leq n$,
\begin{equation}
\mathbf{s}(n,k)=\frac{n!}{k!}\sum_{i_{1}+\dots+i_{k}=n}\frac{(-1)^{n-k}}{i_{1}\dots i_{k}}\mathbb{D}_{i_{1},\dots,i_{k}}^{(n)}.\label{eq:Stirling-first-D-property}
\end{equation}
and 
\begin{equation}
\mathbf{S}(n,k)=\frac{n!}{k!}\sum_{i_{1}+\dots+i_{k}=n}\frac{1}{i_{1}!\dots i_{k}!}\mathbb{D}_{i_{1},\dots,i_{k}}^{(n)}.\label{eq:Stirling-second-D-property}
\end{equation}
Hence, the Stirling operators are continuous (see Proposition 3.7
in \cite{Finkelshtein2022}) and so their adjoints $\mathbf{s}(n,k)^{*}$
and $\mathbf{S}(n,k)^{*}$ are well defined, that is, 
\begin{equation}
\mathbf{s}(n,k)^{*}:(\mathcal{D}_{\mathbb{C}}^{\hat{\otimes}k})'\longrightarrow(\mathcal{D}_{\mathbb{C}}^{\hat{\otimes}n})'\quad\mathrm{and\quad}\mathbf{S}(n,k)^{*}:(\mathcal{D}_{\mathbb{C}}^{\hat{\otimes}k})'\longrightarrow(\mathcal{D}_{\mathbb{C}}^{\hat{\otimes}n})'\label{eq:adjoint-Stirling-operator}
\end{equation}
and satisfy 
\[
\langle w^{(k)},\mathbf{s}(n,k)\varphi^{(n)}\rangle=\langle\mathbf{s}(n,k)^{*}w^{(k)},\varphi^{(n)}\rangle\quad\mathrm{and\quad}\langle w^{(k)},\mathbf{S}(n,k)\varphi^{(n)}\rangle=\langle\mathbf{S}(n,k)^{*}w^{(k)},\varphi^{(n)}\rangle,
\]
for all $w^{(k)}\in(\mathcal{D}_{\mathbb{C}}^{\hat{\otimes}k})'$
and $\varphi^{(n)}\in\mathcal{D}_{\mathbb{C}}^{\hat{\otimes}n}$.
Hence, the Equations \eqref{eq:Stirling-1st-inner-product} and \eqref{eq:Stirling-2nd-inner-product}
imply that 
\begin{align}
(w)_{n} & =\sum_{k=1}^{n}\mathbf{s}(n,k)^{*}w^{\otimes k},\label{eq:falling-factorial-Stirling}\\
w^{\otimes n} & =\sum_{k=1}^{n}\mathbf{S}(n,k)^{*}(w)_{k}.\nonumber 
\end{align}

\begin{prop}[{see \cite[Prop.~3.15]{Finkelshtein2022}}]
 For each $k\in\mathbb{N}$ and $\xi\in\mathcal{D}_{\mathbb{C}},$
\begin{equation}
\sum_{n=k}^{\infty}\frac{1}{n!}\mathbf{S}(n,k)\xi^{\otimes n}=\frac{1}{k!}(\mathrm{e}^{\xi}-1)^{\otimes k}\label{eq:decomposition-g-alpha^k}
\end{equation}
and 
\begin{equation}
\sum_{n=k}^{\infty}\frac{1}{n!}\mathbf{s}(n,k)\xi^{\otimes n}=\frac{1}{k!}(\log(1+\xi))^{\otimes k}.\label{eq:decomposition-alpha^k}
\end{equation}
\end{prop}

\begin{prop}[{see \cite[Prop. 3.19]{Finkelshtein2022}}]
\label{prop:Stirling-1st-2nd-kind-inf-dim}For any $i,n\in\mathbb{N}$,

\[
\sum_{k=1}^{n}\mathbf{s}(k,i)\mathbf{S}(n,k)=\sum_{k=1}^{n}\mathbf{S}(k,i)\mathbf{s}(n,k)=\delta_{n,i}\mathbf{1}^{(i)},
\]
where $\mathbf{1}^{(i)}$ denotes the identity operator on $\mathcal{D}_{\mathbb{C}}^{\otimes i}$.
\end{prop}

\subsection{An Alternative Proof of the Theorem \ref{thm:biorthogonal-property-inf-dim}
(Biorthogonal Property)}

\label{sec:alternative-proof-biorthogonal-property}In Section \ref{sec:Generalized-Dual-Appell-System},
we proved the biorthogonal property of $\mathbb{A}^{\sigma,\beta,\alpha}$
using the definitions of $\mathbb{P}^{\sigma,\beta,\alpha}$ and $\mathbb{Q}^{\sigma,\beta,\alpha}$-systems.
Here, we use a property of the generalized function $Q_{n}^{\sigma,\beta,\alpha}(\Phi^{(n)})$
using the $S_{\pi_{\sigma}^{\beta}}$-transform (see Theorem \ref{thm:generalized-function-S-transform}
below) to provide an alternative proof of the Theorem \ref{thm:biorthogonal-property-inf-dim}.
\begin{lem}
\label{lem:operator-G-on-exp-z-phi}For every $\Phi^{(n)}\in(\mathcal{D}_{\mathbb{C}}^{\hat{\otimes}n})'$,
$z\in\mathcal{D}'_{\mathbb{C}}$ and $\varphi\in\mathcal{D}_{\mathbb{C}}$,
we have 
\[
G(\Phi^{(n)})(\exp\langle z,\varphi\rangle)=\langle\Phi^{(n)},g_{\alpha}(\varphi)^{\otimes n}\rangle\exp\langle z,\varphi\rangle.
\]
In other words, the function $\exp\langle z,\varphi\rangle$ is an
eigenfunction of the generalized function $G(\Phi^{(n)})$.
\end{lem}

\begin{proof}
It follows from \eqref{eq:adjoint-Stirling-operator} that $\mathbf{S}(k,n)^{*}\Phi^{(n)}\in(\mathcal{D}_{\mathbb{C}}^{\hat{\otimes}k})'$.
We use the definition of the differential operator $D(\mathbf{S}(k,n)^{*}\Phi^{(n)})$
to the monomial $\langle z,\varphi\rangle^{m}$, $m\ge k$, to obtain
\begin{align*}
D(\mathbf{S}(k,n)^{*}\Phi^{(n)})\langle z,\varphi\rangle^{m} & =D(\mathbf{S}(k,n)^{*}\Phi^{(n)})\langle z^{\otimes m},\varphi^{\otimes m}\rangle\\
 & =\frac{m!}{(m-k)!}\langle z^{\otimes(m-k)}\hat{\otimes}\mathbf{S}(k,n)^{*}\Phi^{(n)},\varphi^{\otimes m}\rangle\\
 & =\frac{m!}{(m-k)!}\langle z,\varphi\rangle^{m-k}\langle\mathbf{S}(k,n)^{*}\Phi^{(n)},\varphi^{\otimes k}\rangle.
\end{align*}
Now, we apply the above result ($\star$) to the Taylor series of
the function $\exp\langle z,\varphi\rangle$ and obtain 
\begin{align*}
D(\mathbf{S}(k,n)^{*}\Phi^{(n)})(\exp\langle z,\varphi\rangle) & =D(\mathbf{S}(k,n)^{*}\Phi^{(n)})\sum_{m=0}^{\infty}\frac{\langle z,\varphi\rangle^{m}}{m!}\\
 & \overset{\star}{=}\langle\mathbf{S}(k,n)^{*}\Phi^{(n)},\varphi^{\otimes k}\rangle\sum_{m=k}^{\infty}\frac{1}{(m-k)!}\langle z,\varphi\rangle^{m-k}\\
 & =\langle\mathbf{S}(k,n)^{*}\Phi^{(n)},\varphi^{\otimes n}\rangle\exp\langle z,\varphi\rangle.
\end{align*}
Thus, applying the operator $G(\Phi^{(n)})$ to $\exp\langle z,\varphi\rangle$,
we obtain
\begin{align*}
G(\Phi^{(n)})(\exp\langle z,\varphi\rangle) & =\sum_{k=n}^{\infty}\frac{n!}{k!}D(\mathbf{S}(k,n)^{*}\Phi^{(n)})(\exp\langle z,\varphi\rangle)\\
 & =\sum_{k=n}^{\infty}\frac{n!}{k!}\langle\mathbf{S}(k,n)^{*}\Phi^{(n)},\varphi^{\otimes k}\rangle\exp\langle z,\varphi\rangle\\
 & =\langle\Phi^{(n)},g_{\alpha}(\varphi)^{\otimes n}\rangle\exp\langle z,\varphi\rangle,
\end{align*}
where the last equality is a consequence of Equation \eqref{eq:Phi-g-alpha-n-Stirling}.
\end{proof}
\begin{thm}
\label{thm:generalized-function-S-transform}For $\Phi^{(n)}\in(\mathcal{D}_{\mathbb{C}}^{\hat{\otimes}n})'$,
the generalized function $Q_{n}^{\sigma,\beta,\alpha}(\Phi^{(n)})$
satisfies 
\[
S_{\text{\ensuremath{\pi_{\sigma}^{\beta}}}}(Q_{n}^{\sigma,\beta,\alpha}(\Phi^{(n)}))(\varphi)=\langle\Phi^{(n)},g_{\alpha}(\varphi)^{\otimes n}\rangle,\quad\varphi\in\mathcal{V}_{\alpha}\subset\mathcal{D}_{\mathbb{C}}.
\]
\end{thm}

\begin{proof}
Using Lemma \ref{lem:operator-G-on-exp-z-phi}, the $S_{\pi_{\sigma}^{\beta}}$-transform
of $Q_{n}^{\sigma,\beta,\alpha}(\Phi^{(n)})$ is given by 
\begin{align*}
S_{\pi_{\sigma}^{\beta}}(Q_{n}^{\sigma,\beta,\alpha}(\Phi^{(n)}))(\varphi) & =\langle\!\langle G(\Phi^{(n)})^{*}\boldsymbol{1},\mathrm{e}_{\pi_{\sigma}^{\beta}}(\varphi,\cdot)\rangle\!\rangle_{\pi_{\sigma}^{\beta}}\\
 & =\langle\!\langle\boldsymbol{1},G(\Phi^{(n)})\mathrm{e}_{\pi_{\sigma}^{\beta}}(\varphi,\cdot)\rangle\!\rangle_{\pi_{\sigma}^{\beta}}\\
 & ={\displaystyle \frac{1}{l_{\pi_{\sigma}^{\beta}}(\varphi)}\int_{\mathcal{D}'}G(\Phi^{(n)})}(\exp\langle z,\varphi\rangle)\,\mathrm{d}\pi_{\sigma}^{\beta}(z)\\
 & ={\displaystyle \frac{\langle\Phi^{(n)},g_{\alpha}(\varphi)^{\otimes n}\rangle}{l_{\pi_{\sigma}^{\beta}}(\varphi)}\int_{\mathcal{D}'}}\exp\langle z,\varphi\rangle\,\mathrm{d}\pi_{\sigma}^{\beta}(z)\\
 & =\langle\Phi^{(n)},g_{\alpha}(\varphi)^{\otimes n}\rangle.\qedhere
\end{align*}
\end{proof}
Now using the above result, we provide an alternative proof of Theorem
\ref{thm:biorthogonal-property-inf-dim}.
\begin{proof}[Proof of Theorem \ref{thm:biorthogonal-property-inf-dim} \emph{(}Alternative\emph{)}]
The $S_{\pi_{\sigma}^{\beta}}$-transform of $Q_{n}^{\sigma,\beta,\alpha}(\Phi^{(n)})$
at $\alpha(\varphi)$ is given by
\begin{align*}
S_{\pi_{\sigma}^{\beta}}(Q_{n}^{\sigma,\beta,\alpha}(\Phi^{(n)}))(\alpha(\varphi)) & =\langle\!\langle Q_{n}^{\sigma,\beta,\alpha}(\Phi^{(n)}),\mathrm{e}_{\pi_{\sigma}^{\beta}}(\alpha(\varphi),\cdot)\rangle\!\rangle_{\pi_{\sigma}^{\beta}}\\
 & =\sum_{m=0}^{\infty}\frac{1}{m!}\langle\!\langle Q_{n}^{\sigma,\beta,\alpha}(\Phi^{(n)}),\langle C_{m}^{\sigma,\beta},\varphi^{\otimes m}\rangle\rangle\!\rangle_{\pi_{\sigma}^{\beta}}.
\end{align*}
By Theorem \ref{thm:generalized-function-S-transform} with $\varphi$
replaced by $\alpha(\varphi)$ we obtain
\[
S_{\pi_{\sigma}^{\beta}}(Q_{n}^{\sigma,\beta,\alpha}(\Phi^{(n)}))(\alpha(\varphi))=\langle\Phi^{(m)},\varphi^{\otimes m}\rangle.
\]
The result follows by a comparison of coefficients and the polarization
identity.
\end{proof}

\subsubsection*{Acknowledgments}

This work was partially supported by the Center for Research in Mathematics
and Applications (CIMA-UMa) related with the Statistics, Stochastic
Processes and Applications (SSPA) group, through the grant UIDB/MAT/04674/2020
of FCT-Funda{\c c\~a}o para a Ci{\^e}ncia e a Tecnologia, Portugal
and by the Complex systems group of the Premiere Research Institute
of Science and Mathematics (PRISM), MSU-Iligan Institute of Technology.
The financial support of the Department of Science and Technology
-- Accelerated Science and Technology Human Resource Development
Program (DOST-ASTHRDP) of the Philippines under the Research Enrichment
(Sandwich) Program is gratefully acknowledged. Also, special thanks
to Prof. Ludwig Streit for recommending relevant reading resources.

\end{document}